%% file: final_revised1.tex
\documentclass[]{siamart171218}

\usepackage{amsfonts,amsmath,amssymb,bm}
\usepackage{color,cases,float}
\usepackage{graphicx}
\usepackage{hyperref}
\usepackage{psfrag}
\usepackage{eucal}
\usepackage{algorithm}
\usepackage{overpic}
\usepackage{subfigure}
\usepackage{url}
\usepackage[all,cmtip]{xy}
\definecolor{darkblue}{rgb}{0.0,0.0,0.6}
\hypersetup{colorlinks,breaklinks,
	linkcolor=darkblue,urlcolor=darkblue,
	anchorcolor=darkblue,citecolor=darkblue}

 \newtheorem{cor}{Corollary}
 \newtheorem{prop}{Proposition}
 \newtheorem{assum}{Assumption}
 \newtheorem{thrm}{Theorem}
 \newtheorem{lem}{Lemma}

\def\0{{\bf 0}}
\def\1{{\bf 1}}

\def\bes{\begin{equation*}}
	\def\ees{\end{equation*}}
\def\be{\begin{equation}}
	\def\ee{\end{equation}}
\def\bas{\begin{eqnarray*}}
	\def\eas{\end{eqnarray*}}
\def\ba{\begin{eqnarray}}
	\def\ea{\end{eqnarray}}
\def\bma{\begin{bmatrix}}
	\def\ema{\end{bmatrix}}
\def\bmx{\begin{matrix}}
	\def\emx{\end{matrix}}
\def\ben{\begin{enumerate}}
	\def\een{\end{enumerate}}
\def\bit{\begin{itemize}}
	\def\eit{\end{itemize}}
\def\bet{\begin{tabular}}
	\def\eet{\end{tabular}}

\def\e{\epsilon}
\def\de{\delta}

\newcommand{\tat}[1]{{\leavevmode\color{black}#1}}

\def\R{\mathbb{R}}
\def\d{\delta}
\def\la{\langle}
\def\ra{\rangle}
\def\b{{\beta}}
\def\a{\alpha}
\def\g{\gamma}

\def\dist{{\rm dist}}

\begin{document}

\input{ex_shared}

\ifpdf
\hypersetup{
  pdftitle={A Smooth Inexact Penalty Reformulation\\ of Convex Problems
with Linear Constraints},
  pdfauthor={Tatarenko, Nedi\'c}
}
\fi




\maketitle
\begin{abstract}
 In this work, we consider a constrained convex problem with linear inequalities and provide an inexact
penalty re-formulation of the problem.
The novelty is in the choice of the penalty functions, which are
smooth and can induce a non-zero penalty over some points in feasible region of the original constrained problem. The resulting unconstrained penalized problem is parametrized by two penalty parameters
which control the slope and the curvature of the penalty function.
With a suitable selection of these penalty parameters,
we show that the solutions of the resulting penalized unconstrained
problem are  \emph{feasible} for the original constrained problem, under some assumptions.
Also, we establish that, with suitable choices of penalty parameters,
the solutions of the penalized unconstrained problem can achieve a suboptimal value
which is arbitrarily close to the optimal value of the original constrained problem.
For the problems with a large number of linear inequality constraints,
a particular advantage of such a smooth penalty-based reformulation is that
it renders a penalized problem suitable for the implementation
of fast incremental gradient methods, which
require only one sample from the inequality constraints at each
iteration. We consider applying SAGA proposed in~\cite{saga}
to solve the resulting penalized unconstrained problem.
\tat{Moreover, we propose an alternative approach to set up the penalized problem. This approach is based on the time-varying penalty parameters and, thus, does not require knowledge about some problem-specific properties, that might be difficult to estimate. We prove that the single-loop full gradient-based algorithm applied to the corresponding time-varying penalized problem converges to the solution of the original constrained problem in the case of the strongly convex objective function.}
\end{abstract}

\begin{keywords}
  Convex minimization, linear constraints, inexact penalty, incremental methods
\end{keywords}

\begin{AMS}
  90C25, 90C06, 65K05
\end{AMS}

\section{Introduction}
In  this  paper,  we  study  the  problem  of minimizing  a
{\it convex} function $f:\R^n\to\R$ over a convex and closed set $X$
that is the intersection of finitely many {\it convex and closed sets} $X_i$, $i=1,\ldots,m$ ($m\ge 2$ is large),
i.e.,
\ba\label{eq:gprob}
&&\hbox{minimize \ \,} \quad  f(x) \cr
&&\hbox{subject to } \quad  x\in X = \cap_{i=1}^m X_i.
\ea
Throughout the paper, the function $f$ is assumed to be {\it convex} over $\R^n$.
Optimization problems of the form~\eqref{eq:gprob} arise in many areas of research,
such as digital filter settings in communication systems \cite{filter}, energy consumption in Smart Grids \cite{SmartG}, convex  relaxations of  various  combinatorial  optimization  problems in machine learning applications \cite{clustering, matching}.

Our interest is in case when $m$ is large, which prohibits us from using
projected gradient and
augmented Lagrangian methods~\cite{BertsekasConstrOpt},
which require either computation of the (Euclidean) projection or
an estimation of the gradient for the sum of many functions, at each iteration.
To reduce the complexity, one may consider a method that operates on a  single set $X_i$
from the constraint set collection $\{X_1,\ldots,X_m\}$ at each iteration.
Algorithms using random constraint sampling for general convex optimization problems~\eqref{eq:gprob}  have been first considered  in~\cite{Nedich2011} and were extended in~\cite{WangBerts} to a broader class of randomization over the sets of constraints. Moreover, the convergence rate analysis is performed
in~\cite{WangBerts} to demonstrate that the  feasibility  error  diminishes  to  zero  at  a  rate
$O(\log k/ k)$, whereas the optimality error diminishes to zero with the rate of $O(1/\sqrt k)$. For the general convex problems of type~\eqref{eq:gprob},
the latter rate is optimal over the class of optimization methods based on noisy first-order information.

A special case of the problem~\eqref{eq:gprob} with $f\equiv 0$ is a \emph{feasibility} problem,
for which random sampling methods have been considered in~\cite{Polyak2001} for the case of the sets given by convex inequalities,
and in~\cite{CalafiorePolyak2001} for a more specialized case of linear matrix inequalities.
In~\cite{Nedic-cdc-2010}, a connection between the convergence properties of stochastic gradient methods and the existence of solutions for problem~\eqref{eq:gprob} has been studied, and a linear convergence rate has been established for some special cases of the constraint sets $X_i$ (such as those admitting easily computable Euclidean projections). Algorithms with the linear convergence to a solution of feasibility problems defined by a system of linear equations and inequalities have been considered in~\cite{Leventhal, Strohmer}.
An iterated randomized projection scheme for systems of linear equations is proposed in~\cite{Strohmer}, which is a randomized variant of Kaczmarz's  method. This variant employs a single projection per each iteration and is shown to converge with the linear rate that does not depend on the number of equations, but instead, depends on the condition number associated with the linear system of equations.

A possible reformulation of problem~\eqref{eq:gprob} is through the use of the indicator functions of
the constraint sets, resulting in the following unconstrained problem
\be\label{eq:reform}
\min_{x\in\R^n}\sum_{i=1}^m \left\{\frac{1}{m}f(x) + \chi_i(x)\right\},
\ee
where $\chi_i(\cdot):\R^n\to\R\cup\{+\infty\}$ is the indicator function of the set $X_i$
(taking value $0$ at the points
$x\in X_i$ and, otherwise, taking value $+\infty$).
The advantage of this reformulation is that the objective function is the sum of convex functions and incremental methods can be employed that compute only a (sub)-gradient of one of the component functions at each iteration.
The traditional incremental methods do not have memory, and their origin can be traced back to work of Kibardin~\cite{Kibardin}. They have been studied for
smooth least-square problems~\cite{Ber96,Ber97,Luo91}, for training the neural networks~\cite{Gri94,Gri00,LuT94}, for smooth convex problems~\cite{Sol98,Tse98} and
for non-smooth convex problems~\cite{GGM06,HeD09,JRJ09,Kiw04,NeB00,NeB01,NeBBor01,Wright08}
(see~\cite{BertsekasPenalty} for a more comprehensive survey of these methods).
These traditional memoryless incremental methods (randomized and deterministic),
while simple to implement to solve problem~\eqref{eq:reform},
cannot achieve the optimal convergence rate even when $f$ is smooth and strongly convex.
This is due to the non-smoothness of the indicator functions and
the errors that are accumulated during the incremental processing of the functions in the sum~\cite{BertsekasPenalty}.

Reformulation~\eqref{eq:reform} has been considered in~\cite{Kundu2017} as a departure point
toward an exact penalty reformulation using the set-distance functions, thus yielding a penalized problem
of the following form:
\be\label{eq:reformPen}
\min_{x\in\R^n} \left\{f(x) + \lambda h_P(x)\right\},
\ee
where
\[h_P(x)=P(\dist(x,X_1),\ldots,\dist(x,X_m)),\]
with $P$ being some norm in $\R^m$ and $\dist(\cdot,Y)$ being the distance function to a set $Y$.
This exact penalty formulation has been motivated by a simple exact penalty model proposed in~\cite{Bertsekas2011} (using only the set-distance functions)
and a more general penalty model considered in~\cite{BertsekasPenalty}.
In~\cite{Kundu2017}, a lower bound on the penalty level  $\lambda$
has been identified guaranteeing that the optimal solutions of the penalized problem are also optimal solutions of the original problem~\eqref{eq:reform}. However, the proposed approaches in~\cite{Kundu2017} do not utilize
incremental processing, but rather approaches where a full (sub)-gradient of the function
objective in~\eqref{eq:reformPen} is used.

Unlike~\cite{Kundu2017}, our objective in this paper is to consider a penalty-based reformulation of problem~\eqref{eq:gprob} (with linear constraints)
that will allow us to take advantage of the penalized problem structure for the use of incremental methods.
In order to achieve the optimal convergence rates,
we would like to depart from the traditional incremental methods. In particular, we would like to have
a penalty reformulation of problem~\eqref{eq:gprob} that will enable us to employ one of
recently developed {\it fast incremental} algorithms. These algorithms are designed
to solve optimization problems involving a large sum of functions~\cite{saga,svrg,SAG}
which arise in machine learning applications. Unlike the traditional incremental methods that are memoryless,
these fast incremental algorithms require storage of the past (sub)-gradients. Typically, they require storing
the same number $m$ of the (sub)-gradients as the number $m$ of the component functions in the objective\footnote{An approach for memory reduction is proposed, for example in~\cite{svrg}.}.
The stored information is effectively used to control the error due to the incremental processing of the functions,
which in turn allows these algorithms to achieve optimal convergence rates.
A drawback of the fast incremental algorithms, such as SAGA and its various modifications~\cite{Katusha,finito,svrg,scda,accSAGA}, is that
they are not designed to efficiently handle a possibly large number of constraints.
At most, these algorithms allow us to deal with so called
{\it composite optimization} problems, where the composite term corresponds to a regularization function
promoting some special properties of model parameters and has
a simple structure for determining the proximal point~\cite{saga}.

{\it Our focus is on problem~\eqref{eq:gprob} with linear constraints},
\[X_i=\{x\in\R^n \mid \la a_i, x\ra - b_i\le 0\},\]
where $a_i \in \R^n$ and $b_i\in \R$ for all $i=1,\ldots,m$.
Our objective is to develop a penalty model for this problem that will allow us to implement {\it fast incremental} methods~\cite{saga, svrg,SAG} to solve the resulting unconstrained penalized problem.
In order to do so, we will develop a smooth penalty framework motivated by the approach in~\cite{BertsekasPenalty}, and
provide the relations for the solutions of the optimization problem with linear constraints and the solutions of
the corresponding penalized problem.
We consider a penalized reformulation of problem~\eqref{eq:problem} in the following form:
\ba\label{eq:pen-problem0}
&&\hbox{minimize \ \,} \quad  f(x) + \frac{\gamma}{m} \sum_{i=1}^{m} h_\d\left(x; a_i, b_i\right) \cr
&&\hbox{subject to } \quad \ x\in\R^n,
\ea
where
the function $h_\d\left(x; a, b\right)$ is a {\it smooth} penalty function associated
with a linear inequality constraint $\la a,x \ra - b\le 0$,
while $\d\ge0$ and $\gamma>0$ are the penalty parameters. The penalty parameters will control
the slope and the curvature of the penalty function
$\frac{\gamma}{m} \sum_{i=1}^{m} h_\d\left(x; a_i, b_i\right).$
The novelty is in the use of inexact smooth
penalty function $h_\d\left(x; a, b\right)$ that has Lipschitz continuous gradients, which are
not related to the squared set-distance function, which is
in contrast to the inexact distance-based smooth penalties considered in~\cite{Siedlecki}.
Also, this is in contrast with the use of non-smooth exact penalty functions in~\cite{BertsekasPenalty}.
A key property of our penalty framework is its accuracy guarantee, as follows:
For a given accuracy $\delta^0>0$, we show that there exists a range of values
for parameters $\d$ and $\g$ such that
any optimal solution of the penalized problem~\eqref{eq:pen-problem0}
is \emph{feasible} for the original linearly constrained problem.
Moreover, we provide estimates that characterize sub-optimality of the solutions of the penalized problem, i.e.,
we show that the solutions are located within the $\delta^0$-neighborhood of the solutions of the original
constrained problem.

\tat{These properties of the penalized problem allow us to apply any
fast incremental method~\cite{saga, svrg,SAG}. We will employ SAGA to solve the smooth penalized problem
to obtain a suboptimal point
with the sublinear rate $O(1/k)$ in the case of smooth convex function $f$ and
the linear rate $O(q^k)$, with $q<1$, in the case of smooth strongly convex $f$.
However, to guarantee these convergence properties, we need to know some problem specific parameters. These parameters might be difficult to estimate in practice. That is why we also consider an alternative way to define the penalized problem and propose a way to choose time-dependent penalty parameters.

The paper is organized as follows. In Section~\ref{problem-properties}, we formulate the penalized problem,
establish some properties of the chosen penalty function and provide some elementary relation between the penalized problem and the original constrained problem.
In Section~\ref{solution-relation}, we investigate the relation of the solutions of the original problem and its penalized variant. In Section~\ref{sec:incr} we consider applying an existing fast incremental method, namely SAGA, for solving the penalized problem. In Section~\ref{sec:varpar}, we provide an alternative approach to set up time-varying parameters such that the single-loop full gradient-based method applied to the penalized problem converges to the original solution as time runs.
We conclude the paper in Section~\ref{sec:concl}.}

\section{Penalized Problem and its Properties}\label{problem-properties}
We consider the following optimization problem:
\ba\label{eq:problem}
&&\hbox{minimize \ \,} \quad  f(x) \cr
&&\hbox{subject to } \quad  \la a_i,x\ra - b_i\le 0, \  i=1,\ldots,m,
\ea
where the vectors $a_i$, $i=1,\ldots, m$, are nonzero.
We will assume that the problem is {\it feasible}.
Associated with problem~\eqref{eq:problem}, we
consider a penalized problem
\ba\label{eq:pen-problem}
&&\hbox{minimize \ \,} \quad  F_{\g\d}(x) \cr
&&\hbox{subject to }  \quad x\in\R^n,
\ea
where
\begin{equation}\label{eq:penfun}
F_{\g\d}(x) = f(x) + \frac{\gamma}{m} \sum_{i=1}^{m} h_\d\left(x; a_i, b_i\right).\end{equation}
Here, $\gamma>0$ and $\d\ge0$ are penalty parameters.
The vectors $a_i$ and scalars $b_i$ are the same as those characterizing the constraints in problem~\eqref{eq:problem}.
For a given nonzero vector $a\in\R^n$ and $b\in\R$, the penalty function
$h_\d(\cdot;a,b)$ is given by (see also Figure~\ref{fig:penalty})\footnote{A version of the one-sided Huber losses~\cite{Huber}.}
\begin{align}\label{eq:hfun}
h_\delta(x; a,b) = \begin{cases}
\frac{\la a,x\ra -b}{\|a\|}, &\text{ if } \quad \la a,x\ra - b>\de,\\
\frac{(\la a,x\ra -b + \de)^2}{4\de\|a\|}, &\text{ if } \quad -\de\le\la a,x\ra - b\le\de,\\
 0, &\text{ if } \quad \la a,x\ra - b<-\de.
\end{cases}
\end{align}
For any $\d\ge0$, the function $h_\d(x;a,b)$ satisfies the following relations:
\begin{align}\label{eq:hfunineq}
h_\delta(x; a,b) \ge 0\qquad\hbox{for all }x\in\R^n,
\end{align}
\begin{align}\label{eq:hfunineq1}
h_\delta(x; a,b) \le \frac{\delta}{4\|a\|},\qquad\hbox{when }\la a,x\ra\le b,
\end{align}
\begin{align}\label{eq:hfunineq2}
h_\delta(x; a,b) > \frac{\delta}{4\|a\|},\qquad\hbox{when }\la a,x\ra> b.
\end{align}
\begin{figure}[!t]
\centering
\psfrag{-0.5}[c][l]{\scriptsize{$-0.5$}}
\psfrag{0}[c][l]{\scriptsize{$0$}}
\psfrag{0.5}[c][l]{\scriptsize{$0.5$}}
\psfrag{1}[c][l]{\scriptsize{$1$}}
\psfrag{1.5}[c][l]{\scriptsize{$1.5$}}
\psfrag{2}[c][l]{\scriptsize{$2$}}
\psfrag{h}[c][l]{{\Large{$h_{\delta}$}}}
\psfrag{x}[c][b]{\Large{$x$}}
\begin{overpic}[width=1\linewidth]{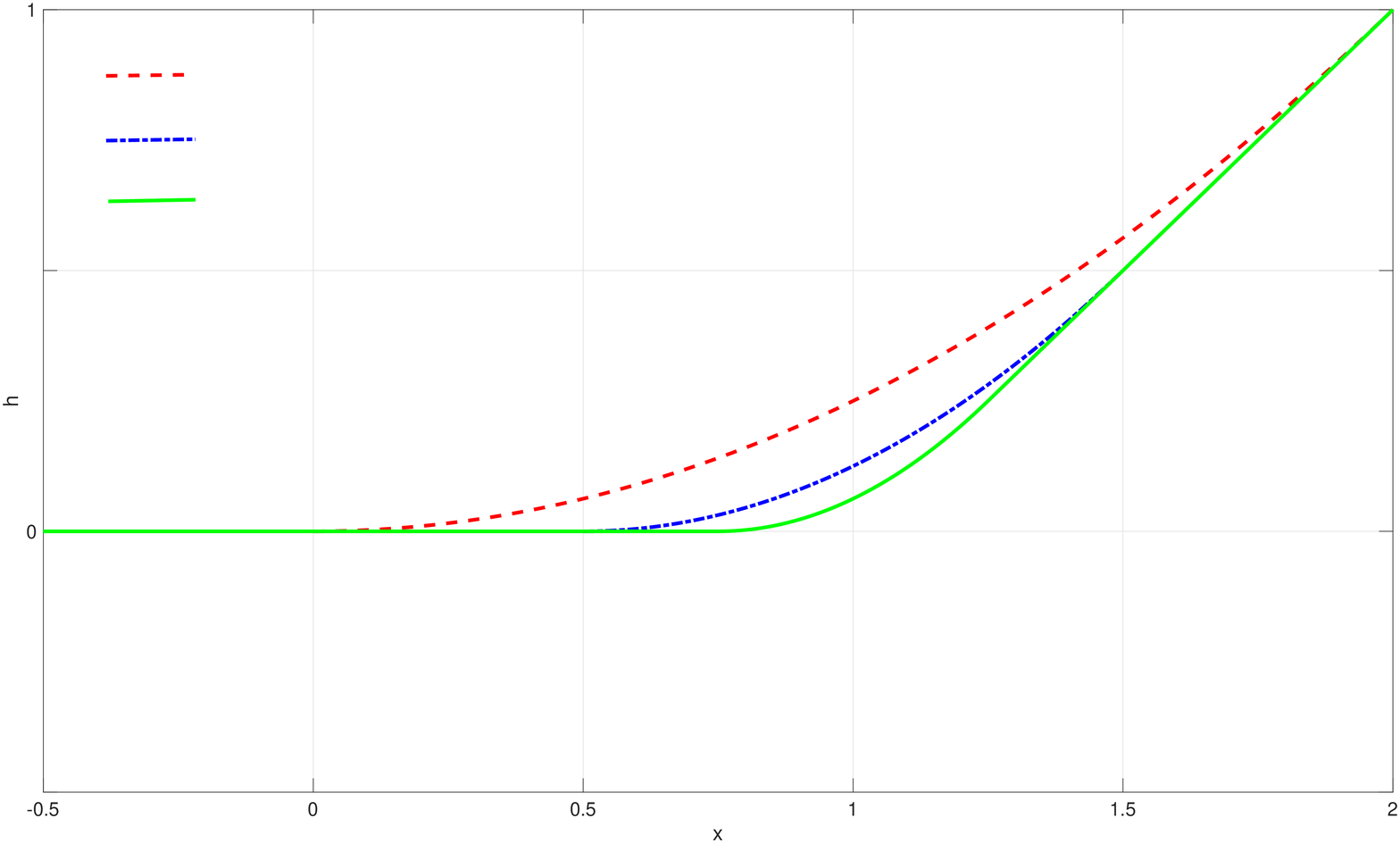}
\put(23,46.5){\scriptsize{$\delta=1$}}
\put(23,43){\scriptsize{$\delta=\frac12$}}
\put(23,39.5){\scriptsize{$\delta=\frac14$}}
\end{overpic}
\caption{Penalty functions $h_\delta(x;1,1)$ for the constraint $x-1\le 0$, $x\in\R$,
with $\d\in\left\{\frac{1}{4},\frac{1}{2},1\right\}$.}
\label{fig:penalty}
\end{figure}

Observe that $h_\delta(x; a,b)$ can be viewed as a composition of
a scalar function
\begin{align}\label{eq:sfun}
p_\d(s)= \begin{cases}
s, &\text{ if } \quad s>\d,\\
\frac{(s + \de)^2}{4\d}, &\text{ if } \quad -\de\le s \le\de,\\
 0, &\text{ if } \quad s<-\de,
\end{cases}
\end{align}
with a linear function  $x\mapsto \la a,x\ra-b$, which is scaled by $\frac{1}{\|a\|}$.
In particular, we have
\be\label{eq:handp}
h_\delta(x; a,b) =\frac{1}{\|a\|}p_\d(\la a,x\ra-b).\ee
The function $p_\d(s)$ is convex on $\R$ for any $\delta\ge0$.
Thus, the function $h_\d(x;a,b)$ is convex on $\R^n$, implying that
the objective function~\eqref{eq:penfun} of the penalized problem~\eqref{eq:pen-problem}
is convex over $\R^n$ for any $\delta\ge0$ and $\g>0$.

Furthermore, observe that the function
$p_\d(\cdot)$ is twice differentiable for any $\d>0$,
with the first and second derivative given by
\begin{align}\label{eq:pderiv}
p'_\d(s)= \begin{cases}
1, &\text{ if } \quad s>\d,\\
\frac{(s + \de)}{2\d}, &\text{ if } \quad -\de\le s \le\de,\\
 0, &\text{ if } \quad s<-\de,
\end{cases}
\end{align}
and
\begin{align*}
p''_\d(s) = \begin{cases}
\frac{1}{2\d}, &\text{ if } \quad -\d\le s\le\d,\\
 0, &\text{ if } \quad s<-\de \quad \text{or}\quad s>\d.
\end{cases}
\end{align*}
Thus, the function $p(s)$ has Lipschitz continuous derivatives with constant $\frac{1}{2\d}$.
Then, the function $h_\delta(\cdot; a,b)$ is differentiable for any $\d>0$ and
its gradient is given by
\be\label{eq:gradh}
\nabla h_\delta(x; a,b) =\frac{1}{\|a\|}\,p'_\d(\la a,x\ra-b) a,\ee
which is Lipschitz continuous with a constant $\frac{\|a\|}{2\d}$,
\be\label{eq:Lipc-gradh}
\|\nabla h_\delta(x; a,b) -\nabla h_\delta(y; a,b)\| \le \frac{\|a\|}{2\d}\,\|x-y\|\qquad\hbox{for all }x,y\in\R^n.\ee
In view of the definition of the penalty function $F_{\g\d}$ in~\eqref{eq:penfun}
and relation~\eqref{eq:gradh},
we can see that the magnitude of the ``slope" of the penalty function is controlled by the parameter $\g>0$,
while the ratio of the parameters $\g$ and $\d$ is controlling the ``curvature" of the penalty function.

\tat{Note that the proposed function $h_\delta(x; a,b)$ penalizes some points in the feasible region. This property allows us to deal with a smooth penalized problem. Moreover, since the penalty function is smaller in the feasible region than outside of it, by adjusting the smoothing parameter $\d$ and the penalty parameter $\gamma$ one can pursue the goal to obtain a feasible point while solving the corresponding penalized problem.
Thus, our choice of the penalty function is motivated by a desire to have the smooth penalized problem~\eqref{eq:pen-problem} with the minimizers being feasible
for the original problem~\eqref{eq:problem}.}
It is worth noting that the penalty function proposed above is a version of the one-sided Huber losses.
Originally, the Huber loss functions were introduced in applications of robust regression models
to make them less sensitive to outliers in data in comparison with the squared error loss~\cite{Huber}.
In contrast,
we use this type of penalty function to smoothen the exact penalties based on the distance to the sets $X_i$
proposed in~\cite{BertsekasPenalty}.
Furthermore, an appropriate choice of the parameter $\delta\ge0$ allows us to overcome the limitation
of the smooth penalties based on the squared distances to the sets $X_i$, which typically
provide an infeasible solution (for the original problem), due to a small penalized value around an optimum lying close to the feasibility set boundary~\cite{Siedlecki}.

In what follows, we let $\Pi_Y [x]$ denote the (Euclidean) projection of a point $x$ on a convex closed set $Y$,
so we have
\[\dist(x, Y) = \|x - \Pi_{Y} [x]\|.\]

The following lemma provides some additional
properties of the penalty function $h_\d(x;a,b)$ that we will use later on.
In fact, the lemma shows stronger results than what we will use, but the results may be of their own interest.

\begin{lem}\label{lem:penalty}
Given a nonzero vector $a\in\R^n$ and a scalar $b\in\R$, consider the penalty function $h_\d(x;a,b)$ defined in~\eqref{eq:hfun} with $\d\ge0$.
Let $Y=\{x\mid \la a,x\ra-b\le0\}.$
Then, we have for $\d=0$,
\[h_0(x;a,b)=\dist(x,Y) \qquad\hbox{for all $x\in\R^n$},\]
and for any $0<\d \le \d'$,
\[h_\d(x;a,b) \le h_{\d'}(x;a,b) \quad\hbox{for all $x\in\R^n$}.\]
\end{lem}
\begin{proof}
Given a vector $x\in\R$, we have
\[\Pi_Y[x]=x-\frac{\max\{\la a,x\ra-b,\,0\}}{\|a\|^2}\, a,\]
so that
\[\dist(x,Y)=\|x-\Pi_Y[x]\|=\frac{\max\{\la a,x\ra-b,\,0\}}{\|a\|}.\]
If $\d=0$, then the last two cases in the definition of  $h_\d(x;a,b)$ reduce to
$h_0(x;a,b)=0$ when $\la a,x\ra-b\le 0$, corresponding to $h_0(x;a,b)=\dist(x,Y)=0$ when $x\in Y$.
When $x\notin Y$,  we have $h_0(x;a,b)=\frac{\la a,x\ra-b}{\|a\|}=\dist(x,Y).$

To prove the monotonicity property, in view of relation~\eqref{eq:handp},
where $p_\d(\cdot)$ is defined in~\eqref{eq:sfun},
it suffices to show that the function $p_\d(\cdot)$ has the monotonicity property, i.e.,
that we have for $0<\d \le \d'$,
\[p_\d(s)\le p_{\d'}(s)\qquad\hbox{for all }s\in\R.\]
To show this let $0<\d\le\d'$.
Note that, for  $s< -\d'$ and $s> \d'$ the functions
$p_\d(\cdot)$ and $p_{\d'}(\cdot)$ coincide, i.e.,
\[p_\d(s)=p_{\d'}(s)\qquad\hbox{when $s< -\d'$ or $s> \d'$}.\]
When $-\d' \le s< -\delta$ we have
\[p_\d(s)=0\le p_{\d'}(s).\]
Next, consider the case when $-\d \le s\le \d$.
Let $s$ be fixed and we view the function $p_\d(s)$ as a function of $\delta$.
For the partial derivative with respect to $\d$, we have
\[\frac{\partial p_\d(s)}{\partial \d}=\frac{1}{4}\,\frac{2(s + \d)\d-(s+\d)^2}{\d^2}
=\frac{1}{4}\, \frac{\d^2-s^2}{\d^2}\ge 0,\]
where the inequality follows by $\d\ge |s|$ and $0<\d$. Thus, $p_\d(s)$ is non-decreasing in $\d$,
implying that $p_\d(s)\le p_{\d'}(s)$. Since $s\in[-\d,\d]$ was arbitrary,
it follows that
\[p_\d(s)\le p_{\d'}(s) \qquad\hbox{when $-\d\le s< \d$}.\]

Finally, let $\d< s\le \d'$, in which case we have
\[p_\d(s)=s=\frac{4s\d'}{4\d'}\le \frac{(s+\d')^2}{4\d'}=p_{\d'}(s),\]
where the inequality is obtained by using $4st\le (s+t)^2$ valid for any $s,t\in\R$.
\end{proof}

In view of Lemma~\ref{lem:penalty}, for the function $F_{\g\d}$ in~\eqref{eq:penfun}
we obtain for any $\g>0$ and any $\d'\ge \d\ge0$,
\[F_{\g\d'}(x)
\ge F_{\g\d}(x)
\ge f(x) + \frac{\gamma}{m} \sum_{i=1}^{m} \dist(x, X_i)\ge f(x)
\qquad\hbox{for all }x\in\R^n.\]
This relation implies an inclusion relation for the level sets of the functions $F_{\g\d}$
and $f$, as given by the following corollary.

\begin{cor}\label{cor:bound}
For any $\g>0$ and for any $t\in\R$, we have
\[\{x\in\R^n\mid F_{\g\d'}(x)\le t\}\subseteq
\{x\in\R^n\mid F_{\g\d}(x)\le t\}\subseteq \{x\in\R^n\mid f(x)\le t\}\]
for all $\d'\ge \d\ge0$.
In particular, if the function $f$ has bounded level sets, then the functions $F_{\d,\g}$ also have bounded level sets
for any $\g>0$ and $\d\ge0$.
\end{cor}

While Corollary~\ref{cor:bound} shows some inclusion relations for the level sets of $F_{\g\d}$ and $f$,
for the same value $t$, it will be important in our analysis to identify a value of $t$  for which
these level sets are nonempty. The following corollary shows that choosing $f(\hat x)$, for any
feasible $\hat x$, can be used to construct non-empty level sets.
\begin{cor}\label{cor:lset}
Let $\g>0$ and $\d\ge 0$ be arbitary, and let
$\hat x$ be a feasible point for the original problem~\eqref{eq:problem}.
Then, for the scalar $t_{\g\d}(\hat x)$ defined by
\[t_{\g\d}(\hat x)= f(\hat x)+\frac{\g\d}{4\min_{1\le i\le m}\|a_i\|},\]
the level set
$\{x\in\R^n\mid F_{\g\d}(x)\le t_{\g\d}(\hat x)\}$ is nonempty and
the solution set $X^*_{\g\d}$ of
the penalized problem~\eqref{eq:pen-problem} is contained in the level set
$\{x\in\R^n\mid f(x)\le t_{\g\d}(\hat x)\}$.
\end{cor}
\begin{proof}
Let $\g>0$ and $\d\ge0$ be arbitrary, and $\hat x$ be any feasible point for the original problem.
Since $\hat x$ is feasible, by relation~\eqref{eq:hfunineq1}, we have
\[h_\delta(\hat x; a_i,b_i) \le \frac{\delta}{4\|a_i\|}\qquad\hbox{for all }i=1,\ldots,m.\]
Therefore,
\[F_{\g\d}(\hat x)
\le f(\hat x)+ \frac{\g\d}{4m} \sum_{i=1}^{m} \frac{1}{\|a_i\|}
\le f(\hat x)+ \frac{\g\d}{4 \min_{1\le i\le m}\|a_i\|}=t_{\g\d}(\hat x),\]
implying that
$\hat x$ belongs to the level set $\{x\in\R^n\mid F_{\g\d}(x)\le t_{\g\d}(\hat x)\}$.
Noting that
\[X^*_{\g\d}\subseteq \{x\in\R^n\mid F_{\g\d}(x)\le t_{\g\d}(\hat x)\},\]
by Corollary~\ref{cor:bound}, we obtain
\[X^*_{\g\d}\subseteq \{x\in\R^n\mid f(x)\le t_{\g\d}(\hat x)\}.\]
\end{proof}

In Corollary~\ref{cor:lset}, the solution set $X^*_{\g\d}$ of
the penalized problem~\eqref{eq:pen-problem} may be empty.
In the next section, we will consider the cases when the solution sets are nonempty
for both the original and the penalized problems.

\section{Relations for Penalized Problem and Original Problem Solutions}\label{solution-relation}
In what follows, we establish some important relations between the solutions of the penalized problem
and the original problem. We assume that the constraint set of problem~\eqref{eq:problem} has a nonempty interior. This assumption implies a special property of the constraint set under consideration which plays a key role in the analysis.
which is valid when the constraint set of problem~\eqref{eq:problem} has a nonempty interior.
To provide this property, we let $X_i$ be the set defined by the $i$th inequality in the constraint set of problem~\eqref{eq:problem}, i.e.,
\[X_i=\{x\in\R^n \mid \la a_i, x\ra - b_i\le 0\},\]
and we define the set $X$ as the intersection of these sets
\[X=\cap_{i=1}^m X_i.\]

We make the following assumption on the interior of the set $X$.
\begin{assum}\label{assum:constrSet}
 The interior of the set $X$ is not empty, i.e.,
 there is a point $\bar x$ such that for some $\epsilon>0$,
 \[\la a_j,\bar x\ra -b_j \le -\epsilon\qquad\hbox{for all $j=1,\ldots,m$}.\]
\end{assum}

In what follows we establish some conditions for solution feasibility of the penalized problem. These conditions involve a constant $\beta$
from the following Hoffman's lemma~\cite{Hoffman}:
\begin{lem}\label{lem:hof}
For the sets $X_i$ there exists
$\beta=\beta(a_1,\ldots,a_m)>0$ such that
\[
\beta \sum_{i=1}^{m} \dist(x, X_i) \ge \dist(x,X)
\]
{for all } $x\in\R^n$.
\end{lem}

We next prove a lemma that will be important for our analysis of
solution feasibility of the penalized problem.
In this lemma and later on, we use the following notation
\be\label{eq:aminmax}
\a_{\min} = \min_{j=1,\ldots,m}\|a_j\|,\qquad \a_{\max} = \max_{j=1,\ldots,m}\|a_j\|.\ee
\tat{The lemma below proves that for any non-feasible point $x$ there exists a feasible point that is not penalized and whose distance to $x$ can be upper bounded by the distance between $x$ and the feasible set plus a term dependent on the problem specific parameters $\b,m,\d,$ and $\a_{\min}$. By an appropriate choice of the smoothness parameter $\d$ we will be able to refine this bound and to use the corresponding result to prove feasibility of the solution for the penalized problem. }

\begin{lem}\label{lem:xin}
 Let Assumption~\ref{assum:constrSet} hold and
 let $\d$ be a positive constant such that $\d\le \e$, where
 $\e$ is defined by Assumption~\ref{assum:constrSet}.
 Then, for any $x\notin X$ there exists a feasible point $x_{in}\in X$ such that
 \begin{enumerate}
  \item[(a)] $h_\d(x_{in};a_j,b_j)=0\quad\hbox{for all } j=1,\ldots,m,$
  \item[(b)] $\|x-x_{in}\|\le \|x- \Pi_X[x]\|+ \frac{\b m \d}{\a_{\min}},$
 \end{enumerate}
where $\a_{\min}$ is defined in~\eqref{eq:aminmax} and $\b$ is Hoffman's constant defined in~Lemma~\ref{lem:hof}.
\end{lem}
\begin{proof}
Let  $0<\delta\le \e$ and consider the perturbed set $X_\d$, which is obtained by perturbing the inequalities by
amount of $\d$ toward the interior of $X$ (see Figure~\ref{fig:radius}), i.e.,
\[X_{\d j} =\{x\in\R^n\mid \la a_j, x\ra -b_j \le -\d\},\qquad X_\d=\cap_{i=1}^m X_{\d i}.\]
Assumption~\ref{assum:constrSet} and the condition $\d\le\e$ imply that  $X_\d\ne\emptyset$.

\begin{figure}[!t]
\centering
\includegraphics[width=0.5\textwidth]{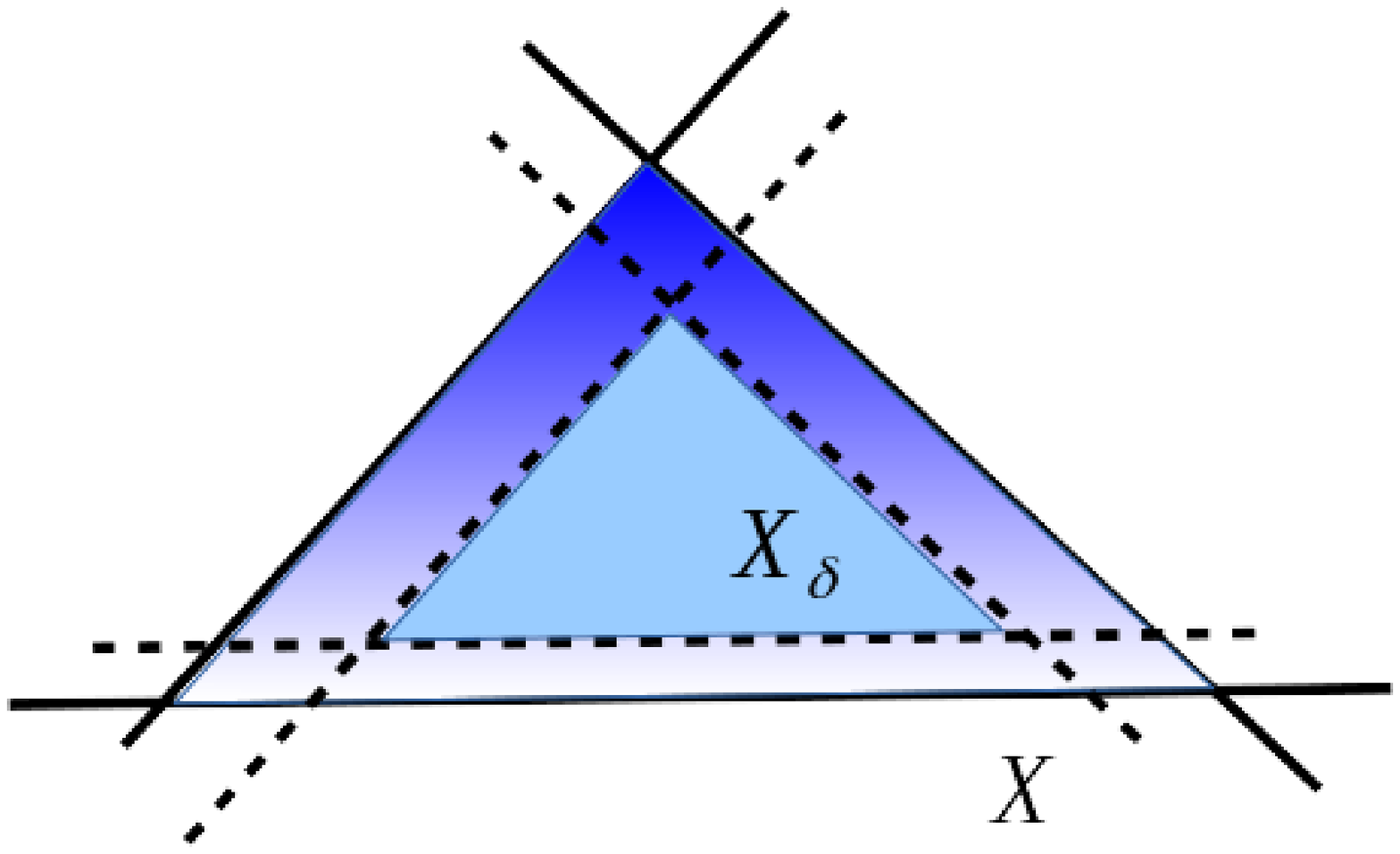}
\caption{Illustration of the set $X_{\d}$.}
\label{fig:radius}
\end{figure}

Let us define
\[x_{in}=\Pi_{X_\d}[x].\]
By the definition of $X_\d$, we have $\la a_j,\Pi_{X_\d}[x]\ra -b_j \le -\d$ for all $j=1,\ldots,m$.
Hence, taking into account the definition of the penalty functions $h_\d(x_{in};a_j,b_j)$, $j=1,\ldots,m,$
(see~\eqref{eq:hfun}),
we obtain
\[h_\d(x_{in};a_j,b_j)=0\qquad\hbox{for all } j=1,\ldots,m,\]
thus showing the relation in part (a).

To estimate  $\|x-x_{in}\|$, let us
consider an intermittent point $\Pi_{X_\d}[\Pi_X[x]]$ obtained by projecting $x$ on $X$
and by projecting the resulting point on the set $X_\d$.
Since $x_{in}=\Pi_{X_\d}[x]$ is the closest point in the set $X_\d$ to $x$,
\be\label{eq:dist-xinx}
\|x-x_{in}\|\le \|x - \Pi_{X_\d}[\Pi_X[x]]\|\le \|x- \Pi_X[x]\|+ \left\|\Pi_X[x] -\Pi_{X_\d}[\Pi_X[x]]\right\|.\ee
Next, note that the constant $\b$ in Hoffman's lemma (see Lemma~\ref{lem:hof}) depends only on
the vectors $a_i$, $i=1,\ldots,m$ (not on the values $b_i$).
Thus, Hoffman's result in Lemma~\ref{lem:hof} applies to the set $X_\d$ with the same constant $\b$ as it holds in respect to the set  $X$, which implies that
\[\left\|\Pi_X[x] -\Pi_{X_\d}[\Pi_X[x]]\right\|=\dist( \Pi_X[x],X_\d)\le \b\sum_{i=1}^m\dist(\Pi_X[x],X_{\d i}).\]

Therefore, according to the definition of $X_{\d j}$, it follows that
\[\left\|\Pi_X[x] -\Pi_{X_\d}[\Pi_X[x]] \right\|\le \b\sum_{j=1}^m\frac{\max\{0,\la a_j, \Pi_X[x]\ra -b_j+\d\}}{\|a_j\|}.\]
Since $\Pi_X[x]\in X$, we have that $\la a_j, \Pi_X[x]\ra -b_j\le 0$ for all $j$. Hence,
\[\left\|\Pi_X[x] -\Pi_{X_\d}[\Pi_X[x]]\right\|\le \b\sum_{j=1}^m\frac{\d}{\|a_j\|}\le \frac{\b m \d}{\a_{\min}}.\]
From the preceding relation and relation~\eqref{eq:dist-xinx} it follows that
\[\|x-x_{in}\|\le \|x- \Pi_X[x]\|+ \frac{\b m \d}{\a_{\min}},\]
thus establishing the result in part (b).
\end{proof}

We next turn our attention to the solution sets of the problems.
We let
$X^*$ and $X^*_{\g\d}$ denote the solution sets of the original problem and the penalized problem, respectively,
i.e.,
\[X^*=\left\{ x\in X \mid f(x)=\min_{z\in X} f(z) \right\},\qquad
X^*_{\g\d}=\left\{ x\in \R^n \mid F_{\g\d}(x)=\min_{z\in \R^n} F_{\g\d}(z) \right\}.\]

In our main result establishing that $X^*_{\g\d}\subseteq X$, under some conditions on the penalty parameters
$\g$ and $\d$,
we will require that the function $f$ has uniformly bounded subgradients over a suitably defined region.
If the constraint set $X$ is bounded, then the set $X$ can be taken as such a region and an upper bound
for the subgradient norms can be defined by
\[L=\max\{\|s\|\mid s\in\partial f(x), \ x\in X\},\]
where $\partial f(x)$ is the subdifferential set of $f$ at $x$.
If $X$ is unbounded, we identify a suitable region in the following lemma. In particular, the region
should be large enough to contain the sets $X^*_{\g\d}$ for a range of penalty values,
and also the points $x_{in}$ from Lemma~\ref{lem:xin}(b) for each $x\in X^*_{\g\d}$.

\begin{lem}\label{lem:Lipconst}
Let Assumption~\ref{assum:constrSet} hold and
 let $\d$ be a positive constant such that $\d\le \e$, where
 $\e$ is defined by Assumption~\ref{assum:constrSet}.
 Assume that $f$ has bounded level sets.
 Then, for all $\g>0$ and $\d>0$ satisfying $\g\d\le c$ for some $c>0$,
 there is a ball centered at the origin that contains  all the points
 $\Pi_X[x]$ with $x\in X^*_{\g\d}$ and the points $x_{in}$ satisfying Lemma~\ref{lem:xin}(b) with
 $x\in X^*_{\g\d}$.
 The radius of this ball depends on some feasible point $\hat x\in X$,
 the given value of $c$, the value $\e$ from Assumption~\ref{assum:constrSet},
 and the problem characteristics reflected in the constants $\a_{\min}$, $m$ and $\b$ from Hoffman's result
 (see~Lemma~\ref{lem:hof}).
\end{lem}
\begin{proof}
Since $f$ has bounded level sets, by Corollary~\ref{cor:bound}, the functions $F_{\g\d}$ also have bounded
level sets for all $\d\ge0$ and $\g\ge0$. Hence,
the solution set $X^*$ is nonempty and, also, the solution sets $X^*_{\g\d}$ that
are  nonempty for all $\g>0$ and $\d>0$.
We next employ Corollary~\ref{cor:lset} to construct a compact set that contains
the optimal sets $X^*_{\g\d}$
are  nonempty for all $\g>0$ and $\d>0$ for a range of values of these penalty parameters.

To start, we choose some feasible point $\hat x\in X$ and, by Corollary~\ref{cor:lset},
we obtain
\[X^*_{\g\d}\subseteq \{x\in\R^n\mid f(x)\le t_{\g\d}(\hat x)\}\qquad\hbox{for all $\d\ge0$ and $\g>0$},\]
where
\[t_{\g\d}(\hat x)=f(\hat x)+\frac{\g\d}{4\a_{\min}}.\]
Under the assumption that $\g\d\le c$ for some $c>0$,
we have $t_{\g\d}\le f(\hat x) + \hat c,$
where $\hat c=\frac{c}{4\a_{\min}}$ (see~\eqref{eq:aminmax} for the definition of $\a_{\min}$).
Thus, we consider the level set
\[ \{x\in\R^n\mid f(x)\le f(\hat x)+ \hat c\},\]
which is bounded by the assumption that $f$ has bounded level sets.
Furthermore,
\[X^*_{\g\d}\subseteq \{x\in\R^n\mid f(x)\le f(\hat x)+ \hat c\}\quad
\hbox{for all $\g>0$ and $\d>0$ satisfying $\g\d\le c$}.\]
Hence, these optimal sets are uniformly bounded, i.e., for some constant $B_1(\hat x,\hat c)>0$,
\[\|x\|\le B_1(\hat x,\hat c)\qquad\hbox
{for all $x\in X_{\g\d}^*$, and all $\g>0$ and $\d>0$ with $\g\d\le c$}.\]
Since the projection operator is non-expansive, the projections of the points in the set $X^*_{\g\d}$
on the set $X$ are also bounded,
i.e., for some constant $B_2(\hat x,\hat c) = B_2(B_1(\hat x,\hat c))>0$,
\be\label{eq:prxbound}
\|\Pi_X[x]\|\le B_2(\hat x,\hat c)
\qquad\hbox{for all $x\in X_{\g\d}^*$, and all $\g>0$ and $\d>0$ with $\g\d\le c$}.\ee
Finally, for each $x\in X^*_{\g\d}$, consider a point $x_{in}$ as given in Lemma~\ref{lem:xin}.
Then, by Lemma~\ref{lem:xin}(b) for each $x\in X^*_{\g\d}$, it follows that
\[\|x_{in}\|\le \| x_{in}- x\| +  \|x\|
\le \|x - \Pi_X[x]\| + \frac{\b m\d}{\a_{\min}} + \|x\|
\le 2\|x\| +\|\Pi_X[x]\| + \frac{\b m\e}{\a_{\min}}, \]
where we use assumption that $\d\le\e$.
Thus, for each $x\in X^*_{\g\d}$, the point $x_{in}$ from Lemma~\ref{lem:xin}(b) satisfies
the following relation
\[\|x_{in}\|\le B(\hat x,\hat c,\e)
\qquad\hbox{for all $\g>0$ and $\d>0$ with $\d\le \e$ and $\g\d\le c$},\]
where
\[B(\hat x,\hat c,\e)= 2B_1(\hat x,\hat c)+B_2(\hat x,\hat c)+\frac{\b m\e}{\a_{\min}}.\]
In view of~\eqref{eq:prxbound}, the ball centered at the origin with the radius $B(\hat x,\hat c,\e)$
also contains $\Pi_X[x]$ for all $x\in X_{\g\d}^*$ and for all $\g>0$ and $\d>0$, with $\d\le \e$ and $\g\d\le c$.
Since $\hat c=\frac{c}{4\a_{\min}}$, we see that the constant $B(\hat x,\hat c,\e)$ depends
on the choice of the feasible point $\hat x\in X$, the given value of $c$, the value $\e$ from Assumption~\ref{assum:constrSet},
and the problem characteristics reflected in the constants $\a_{\min}$, $m$ and $\b$ from Hoffman's result
(see~Lemma~\ref{lem:hof}).
\end{proof}

In what follows, we will let $R(c,\e)$ denote the radius of the ball identified in Lemma~\ref{lem:Lipconst},
and suppress the dependence on the other parameters. We define
\be\label{eq:Lipc}
L(c,\e)=\max\{\|s\|\mid s\in\partial f(x),\, \|x\|\le R(c,\e)\}.\ee
\tat{Note that the constant $L(c,\e)$ defined above is a bound for the subgradients of $f$ in the ball with the center at the origin and the radius $R(c,\e)$. The main properties of this ball are formulated by Lemma~\ref{lem:Lipconst}. The parameters defining $R(c,\e)$ include $\e$ from Assumption~\ref{assum:constrSet} and a predefined constant $c$ which controls the relation between $\d$ and $\g$, namely $\g\d<c$, and, thus, guarantees boundedness of the optimal sets of the penalized problems given such $\g$ and $\d$ (see proof of Lemma~\ref{lem:Lipconst}).}

\tat{With Lemma~\ref{lem:xin} and Lemma~\ref{lem:Lipconst},
we are ready to provide a key relation for the solutions of the penalized problem and the original problem. Before stating this result, we emphasize one time more the meaning of the parameters used in the analysis. Specifically, $\d$ and $\g$ are used to define the penalized problem that can be solved by unconstrained optimization methods. The parameter $\d$ is chosen to define the penalty functions. These functions penalize a region in the feasible set. Such choice of penalties is motivated by the goal to be able to apply fast incremental optimization procedures to a smooth unconstrained penalized problem. The parameter $\g$ defines the weight of the penalty functions. By increasing $\g$ we increase cost for constraints violation and, thus, let any solution of the penalized problem approach the feasible region.
Next, we show that for sufficiently small values of the smoothing parameter $\d$ and sufficiently large values of the penalty parameter $\g$, the solutions of the penalized problem are feasible for the original problem. The choice of these values depends on the number of constraints $m$, $\a_{\min}$ and $\a_{\max}$, Hoffman's constant $\beta$, the parameter $\e$ from Assumption~\ref{assum:constrSet}, as well as $L(c,\e)$ defined in \eqref{eq:Lipc}.}

\begin{prop}\label{prop:penSol}
  Let
  Assumption~\ref{assum:constrSet} hold and
  assume that $f$ has bounded level sets.
  Let the parameters $\gamma$ and $\d$ be chosen such that
  \[0<\d<\min\left\{\e,\frac{16\a_{\min}^2}{\b^2 m^2}\right\}, \qquad \g\d\le c,\qquad \g\ge \Gamma,\]
  with
    \[\Gamma =\max\left\{L\left(\frac{1}{m\beta} - \frac{\sqrt{\de}}{4\a_{\min}}\right)^{-1},4mL\a_{\max} \left( \frac{1}{\sqrt{\d}} + \frac{\b m}{\a_{\min}}\right)\right\},\]
   where $c>0$ is arbitrary, $\e$ is the constant from Assumption~\ref{assum:constrSet},
   $\beta$ is the constant from Hoffman's bound (see~Lemma~\ref{lem:hof}),
   the scalars $\a_{\min}$ and $\a_{\max}$ are defined in~\eqref{eq:aminmax},
   while $L=L(c,\e)$ is defined by~\eqref{eq:Lipc}.
   Then, every point in the solution set $X^*_{\g\d}$ of the penalized problem is feasible
    for the problem~\eqref{eq:problem}, namely $X^*_{\g\d}\subset X$.
 \end{prop}
 \tat{\begin{remark}\label{rem:param1}
 Note that the conditions for $\d$ and $\g$ in Proposition~\ref{prop:penSol} are not contradictive. Indeed, $\Gamma\le\frac{c}{\d}$ for $\delta$ such that
 \begin{align}\label{eq:param1}
   \sqrt{\d} \le \min\left\{\frac{\sqrt{c^2 + \frac{64\a^2_{\min}Lc}{m\b}}-c}{8\a_{\min}L}, \frac{\sqrt{\a^2_{\min}+\frac{c\b\a_{\min}}{L\a_{\max}}}-\a_{\min}}{2m\b}\right\}.
 \end{align}
 The inequality $\Gamma\le\frac{c}{\d}$ implies existence of $\g$ such that $\Gamma\le\g\le\frac{c}{\d}$ holds for any choice of $c$.
 \end{remark}}
 \begin{proof}
 	Since $f$ has bounded level sets, the solution set $X^*$ and the solution sets $X^*_{\g\d}$
	are nonempty for all $\g>0$ and $\d>0$.
	To arrive at a contradiction, let us assume that there exists some $\g$ and $\d$
	satisfying the conditions in the proposition and that $X^*_{\g\d}\not\subseteq X$.
	Thus, there exists a solution $x^*_{\g\d}\in X^*_{\g\d}$ and $x^*_{\g\d}\not \in X$.
	Define
	\[\hat x^*_{\g\d}=\Pi_{X} [x^*_{\g\d}].\]
	We consider two possibilities:
	$\|\hat x^*_{\g\d} - x^*_{\g\d}\|\ge \sqrt{\delta}$ and  $\|\hat x^*_{\g\d} - x^*_{\g\d}\|< \sqrt{\delta}$.
	
	{\it Case 1: $\|\hat x^*_{\g\d} - x^*_{\g\d}\| \ge \sqrt{\delta}$}.
	By Lemma~\ref{lem:penalty} we have that $h_{\de}(x;a_i,b_i)\ge\dist(x, X_i)$ for all $i=1,\ldots,m$.
	Thus, by the definition of the functions $F_{\g\d}$, for any $x\in\R^n$ we can write
	\[F_{\g\d}(x) \ge f(x) + \frac{\gamma}{m} \sum_{i=1}^{m} \dist(x, X_i).\]
	Then, by Hoffman's lemma (see~Lemma~\ref{lem:hof}), for some $\beta>0$ we have
	\[F_{\g\d}(x) \ge f(x) + \frac{\gamma}{m\beta} \dist(x, X).\]
	Letting $x=x^*_{\g\d}$ in the preceding relation, we obtain
	\begin{align*}
	F_{\g\d}(x^*_{\g\d})
	& \ge f(x^*_{\g\d})
	+ \frac{\gamma}{m\beta} \|\hat x^*_{\g\d} - x^*_{\g\d}\| +f(\hat x^*_{\g\d})-f(\hat x^*_{\g\d}) \cr
	& \ge  \frac{\gamma}{m\beta} \|\hat x^*_{\g\d} - x^*_{\g\d}\|
	- L\|\hat x^*_{\g\d} - x^*_{\g\d}\| +f(\hat x^*_{\g\d}) \cr
	& = \left(\frac{\gamma}{m\beta}  - L \right)\|\hat x^*_{\g\d} - x^*_{\g\d}\|
	+F_{\g\d}(\hat x^*_{\g\d}) - \frac{\gamma}{m}\sum_{i=1}^{m} h_{\d}(\hat x^*_{\g\d};a_i,b_i),
        \end{align*}
        where in the second inequality we use the assumption that the norms of
        the subgradients in the subdifferential set $\partial f(x)$ are bounded by $L$ in a region
        containing the points $x=\hat x^*_{\g\d}$ and $x^*_{\g\d}$ (see Lemma~\ref{lem:Lipconst} and~\eqref{eq:Lipc}).
        Taking into the account that $h_\d(x;a_i,b_i)\le \frac{\d}{4\|a_i\|}$ when $x\in X_i$
        (see inequality~\eqref{eq:hfunineq1} and the definition of the set $X_i$)
        and using $\hat x^*_{\g\d}\in X\subseteq X_i$, we see that
        \[F_{\g\d}(x^*_{\g\d}) \ge
        \left(\frac{\gamma}{m\beta}  - L \right)\|x^*_{\g\d} - \hat x^*_{\g\d}\|
         +F_{\g\d}(\hat x^*_{\g\d}) - \frac{\gamma\d}{4m}\sum_{i=1}^{m}\frac{1} {\|a_i\|}.\]
         Note that the condition $\g\ge \Gamma$ and the definition of $\Gamma$ imply
         $\gamma\ge Lm\beta$\footnote{Indeed, $\g\ge \Gamma\ge 4mL\a_{\max} \left( \frac{1}{\sqrt{\d}} + \frac{\b m}{\a_{\min}}\right)\ge4m^2L\b\ge Lm\beta$.}. Using  the relations $\gamma\ge Lm\beta$
         and $\|\hat x^*_{\g\d} - x^*_{\g\d}\|\ge\sqrt{\delta}$, which we assumed,
         we further obtain
         \begin{align}\label{eq:caseone}
         F_{\g\d}(x^*_{\g\d}) >\left(\frac{\gamma}{m\beta}  - L\right)\sqrt{\d} -  \frac{\g\d}{4\a_{\min}}
         + F_{\g\d}(\hat x^*_{\g\d})\ge F_{\g\d}(\hat x^*_{\g\d}),
         \end{align}
         where the last inequality is obtained by using
         $\left(\frac{\gamma}{m\beta}  - L\right)\sqrt{\d} -  \frac{\g\d}{4\a_{\min}} \ge 0$, which is equivalent to
         \[\frac{\gamma}{m\beta}  - L-\frac{\g\sqrt{\d}}{4\a_{\min}}\ge 0\quad\iff\quad
         \g \left(\frac{1}{m\beta}-\frac{\sqrt{\d}}{4\a_{\min}}\right)\ge L. \]
         The last inequality holds
         due to the conditions that we imposed on the parameters $\gamma$ and $\delta$, namely, that
         $\d < \frac{16\a_{\min}^2}{\b^2 m^2}$ and
          $\gamma\ge L\left(\frac{1}{m\beta} - \frac{\sqrt{\d}}{4\a_{\min}}\right)^{-1}$.
          Thus, relation~\eqref{eq:caseone} implies that $F_{\g\d}(x^*_{\g\d})>F_{\g\d}(\hat x^*_{\g\d})$,
          which contradicts the fact that $x_{\g\d}^*$ is an unconstrained minimizer of $F_{\g\d}$.

          {\it Case 2: $\|\hat x^*_{\g\d} - x^*_{\g\d}\|< \sqrt{\d}$}.
          Since $x^*_{\g\d}\notin X$, under Assumption~\ref{assum:constrSet} and the condition $\d\le\e$,
          we can apply Lemma~\ref{lem:xin} with $x=x^*_{\g\d}$.
          According to Lemma~\ref{lem:xin}, there exists a feasible point $x'_{\g\d}\in X$ such that
          \begin{align}\label{eq:pennonfeas1}
          h_{\d}(x'_{\g\d};a_i,b_i)= 0 \qquad \mbox{ for all $i=1,\ldots, m$,}
          \end{align}
          and
          \begin{align}\label{eq:dist}
          \|x^*_{\g\d} - x'_{\g\d}\|\le \| x^*_{\g\d} - \Pi_X[x^*_{\g\d}]\|+ \frac{\b m \d}{\a_{\min}}.
          \end{align}
          Using the point $x'_{\g\d}$, we have
          \begin{align}\label{eq:ineqfinal1}
          F_{\g\d}(x^*_{\g\d}) - F_{\g\d}(x'_{\g\d})
          & = f(x^*_{\g\d})- f(x'_{\g\d})
          + \frac{\gamma}{m}\left( \sum_{i=1}^{m}h_{\d}(x^*_{\g\d};a_i,b_i) - h_{\d}(x'_{\g\d};a_i,b_i)\right)\cr
          & \ge   - L \|x^*_{\g\d} - x'_{\g\d}\|
          + \frac{\gamma}{m} \sum_{i=1}^{m}h_{\d}(x^*_{\g\d};a_i,b_i),
          \end{align}
          where we use the assumption that $f$ has bounded subgradients over the region
          containing the point $x'_{\g\d}$ (see Lemma~\ref{lem:Lipconst} and~\eqref{eq:Lipc})
          and relation~\eqref{eq:pennonfeas1}.
          Since $x^*_{\g\d}\notin X$, there exists a constraint $j$ that is violated at $x^*_{\g\d}$, i.e., we have
          \[\la a_j, x^*_{\g\d}\ra-b_j>0.\]
          For the violated constraint $j$, by property~\eqref{eq:hfunineq2}, for the penalty function
          $h_\d(\cdot; a_j,b_j)$, we have \begin{align}\label{eq:pennonfeas}
          h_{\d}(x^*_{\g\d};a_j,b_j) > \frac{\d}{4\|a_j\|}\ge \frac{\d}{4\a_{\max}}.
          \end{align}
          Using~\eqref{eq:pennonfeas} and the fact that the penalty functions are non-negative
          (see~\eqref{eq:hfunineq}),
          we obtain
          \begin{align}\label{eq:pennonfeas2}
          \sum_{i=1}^{m}h_{\d}(x^*_{\g\d};a_i,b_i)\ge h_{\d}(x^*_{\g\d};a_j,b_j) >  \frac{\d}{4\a_{\max}}.
           \end{align}
           Substituting estimate~\eqref{eq:pennonfeas2} in relation~\eqref{eq:ineqfinal1} we further obtain
           \begin{align}\label{eq:ineqfinal2}
           F_{\g\d}(x^*_{\g\d}) - F_{\g\d}(x'_{\g\d})
           &> - L \|x^*_{\g\d} - x'_{\g\d}\|+\frac{\g\d}{4m\a_{\max}}\cr
           &\ge - L\left(\|x^*_{\g\d} - \Pi_X[x^*_{\g\d}]\|+ \frac{\b m \d}{\a_{\min}}\right) + \frac{\g\d}{4m\a_{\max}},
         \end{align}
         where the last inequality is obtained by using~\eqref{eq:dist}.
         Since  $\hat x^*_{\g\d}= \Pi_X[x^*_{\g\d}]$ and we work under the assumption that
         $\|\hat x^*_{\g\d} - x^*_{\g\d}\|< \sqrt{\d}$, from~\eqref{eq:ineqfinal2}
         we have
         \be\label{eq:ineqfinal} F_{\g\d}(x^*_{\g\d}) - F_{\g\d}(x'_{\g\d})
         > -L\left(\sqrt{\d}+ \frac{\b m \d}{\a_{\min}}\right) + \frac{\gamma\d}{4m\a_{\max}} \ge0,\ee
         where the last inequality is due to the conditions imposed on $\d$ and $\gamma$,
         namely, the condition that
         $\g\ge4mL\a_{\max}\left( \frac{1}{\sqrt{\d}} + \frac{\b m}{\a_{\min}}\right)$,
         which is equivalent to $\frac{\g\d}{4m\a_{\max}} - L\left(\sqrt{\d}+ \frac{\b m \d}{\a_{\min}}\right) \ge0$.
         Hence, it follows that
         \[F_{\g\d}(x^*_{\g\d}) - F_{\g\d}(x'_{\g\d})>0,\]
         which contradicts the fact that $x_{\g\d}^*$ is an unconstrained minimizer of $F_{\g\d}$.
   \end{proof}

   Let us note that, for the constant $\Gamma$ in Proposition~\ref{prop:penSol}, we have $\Gamma>0$
   in view of the condition
   \[ 0<\d<\frac{16\a_{\min}^2}{\b^2 m^2}.\]
   The condition $\g\d\le c$ in Proposition~\ref{prop:penSol} is imposed only to ensure
   the existence of the subgradient norm bound $L(c,\e)$.
   \tat{In view of Remark~\ref{rem:param1}, one way to think about the choices of $\g$ and $\d$ that satisfy the conditions in
   Proposition~\ref{prop:penSol} is as follows. We first select a large $c$ and determine an estimate $\hat L\ge L(c,\e)$.
   Having $\hat L$, we choose  a penalty value $\d>0$ that satisfies
   $\d<\min\left\{\e,\,\frac{16\a_{\min}^2}{\b^2 m^2}\right\}$ and \eqref{eq:param1}, where $L$ is replaced by $\hat L$.
   Then we set up $\g$ such that $\d\le\g\le \frac{c}{\d}$ holds.}

   Now, we provide a relation between optimal values for the original and penalized problems.
   We consider the cases when $f$ is strongly convex and non-strongly convex, separately.

   The following proposition establishes a key relation between solutions
   $x^*_{\g\d}$ and $x^*$ for the case when $f$ is strongly convex.
   In particular, the proposition provides a set of conditions on the parameters
   $\d$ and $\g$ ensuring that  the distance between $x^*_{\g\d}$ and $x^*$ does not exceed
   a desired accuracy $\d^0$, i.e., $\|x^*_{\g\d} - x^*\|^2\le \d^0$.

   \begin{prop}\label{prop:strong}
   Let $\d^0$ be a given accuracy parameter. Let Assumption~\ref{assum:constrSet} hold and
   let $f$ be strongly convex with a constant $\mu_f>0$.
   Let the parameters $\gamma$ and $\d$ be chosen such that
   \[0<\d < \min\left\{\e,\frac{16\a_{\min}^2}{\b^2 m^2}\right\},\qquad
    \Gamma \le\g\le \frac{2\mu_f \a_{\min}\delta^0}{\de}, \]
    with
    \[\Gamma
    =\max\left\{L\left(\frac{1}{m\beta} - \frac{\sqrt{\de}}{4\a_{\min}}\right)^{-1}, \,4mL\a_{\max}\left( \frac{1}{\sqrt{\d}} + \frac{\b m}{\a_{\min}}\right)\right\},\]
   where $\e$ is the constant from Assumption~\ref{assum:constrSet},
   $\beta$ is the constant from Hoffman's bound (see~Lemma~\ref{lem:hof}), the scalars
   $\a_{\min}$ and $\a_{\max}$ are defined in~\eqref{eq:aminmax},  while
   $L=L(c,\e)$ is the bound on
   the subgradient norms as defined in~\eqref{eq:Lipc} with $c>2\mu_f \a_{\min}\delta^0$.
   Then, the original problem~\eqref{eq:problem} and
   the penalized problem~\eqref{eq:pen-problem}
   have unique solutions, $x^*$ and $x^*_{\g\d}$, respectively, which satisfy the following relation:
   \[\|x^*_{\g\d} - x^*\|^2\le \de^0.\]
 \end{prop}
  \tat{\begin{remark}\label{rem:param2}
 Note that the conditions for $\d$ and $\g$ in Proposition~\ref{prop:strong} are not contradictive. Indeed, $\Gamma\le\frac{2\mu_f \a_{\min}\delta^0}{\d}$ for $\delta$ such that
 \begin{align}\label{eq:param2}
   \sqrt{\d} \le \min\left\{\frac{\sqrt{\tilde{c}^2 + \frac{64\a^2_{\min}L\tilde{c}}{m\b}}-\tilde{c}}{8\a_{\min}L}, \frac{\sqrt{\a^2_{\min}+\frac{\tilde{c}\b\a_{\min}}{L\a_{\max}}}-\a_{\min}}{2m\b}\right\}, \mbox{ where }\tilde{c} = 2\mu_f \a_{\min}\delta^0.
 \end{align}
 The inequality $\Gamma\le\frac{2\mu_f \a_{\min}\delta^0}{\d}$ implies existence of $\g$ such that $\Gamma\le\g\le\frac{2\mu_f \a_{\min}\delta^0}{\d}$ holds for any choice of $\d^0$.
 \end{remark}}
 \begin{proof}
 Since the function $f:\R^n\to\R$ is strongly convex with a constant $\mu_f>0$,
   by the convexity of the penalty function $h_\delta$, the penalized objective function
   $F_{\g\d}$ in~\eqref{eq:penfun}
   is also strongly convex with the same strong convexity constant $\mu_f$,
   for any $\g\ge 0$. Hence, the original problem~\eqref{eq:problem} and
   the penalized problem~\eqref{eq:pen-problem}
   have unique solutions, denoted respectively by
   $x^*$ and $x^*_{\g\d}$.

   By the relations $c>2\mu_f \a_{\min}\delta^0$ and $\g\le \frac{2\mu_f \a_{\min}\delta^0}{\de}$,
   it follows that $\g\d\le c$. Thus, the conditions of Proposition~\ref{prop:penSol} are satisfied.
   According to Proposition~\ref{prop:penSol},
   the vector $x^*_{\g\d}$ is feasible i.e., $x^*_{\g\d}\in X$, implying that
   \be\label{eq:optvals}
   f(x^*)\le f(x^*_{\g\d}).\ee
   Since the penalty functions are non-negative, we have $h_{\d}(x^*_{\g\d};a_i,b_i)\ge 0$ for all
   $i=1,\ldots,m$ (see~\eqref{eq:hfunineq}). The point $x^*$ is feasible but it may be penalized, in which case
   $h_{\d}(x^*;a_i,b_i)\le \frac{\d}{4\|a_i\|}$ for all $i=1,\ldots,m$ (see~\eqref{eq:hfunineq1}).
   Therefore, we have
   \be\label{eq:pendiff}
   h_{\d}(x^*;a_i,b_i) - h_{\d}(x^*_{\g\d};a_i,b_i)\le \frac{\d}{4\|a_i\|}\qquad\hbox{for all $i=1,\ldots,m$}.\ee
   Using the relations~\eqref{eq:optvals} and~\eqref{eq:pendiff}, we obtain
   \begin{align*}
   F_{\g\d}(x^*) &- F_{\g\d}(x^*_{\g\d})\cr& = f(x^*)- f(x^*_{\g\d})
   + \frac{\g}{m}\left( \sum_{i=1}^{m}h_{\d}(x^*;a_i,b_i)- h_{\d}(x^*_{\g\d};a_i,b_i)\right)
    \le  \frac{\g\d}{4\a_{\min}}.
    \end{align*}
    By the strong convexity of $F_{\g\d}$, it follows that
    \begin{align}\label{eq:strconv}
     \|x^* - x^*_{\g\d}\|^2 \le \frac{2}{\mu_f}(F_{\g\d}(x^*) - F_{\g\d}(x^*_{\g\d}))
     \le \frac{\g \de}{2\mu_f\a_{\min}}\le \de^0,
     \end{align}
     where the last inequality in the preceding relation is
     due to the choice of $\gamma\le\frac{2\mu_f \a_{\min}\delta^0}{\de }$.
 \end{proof}

     By slightly adapting the choice of $\g$, we can provide an estimate
     for the  function value $f(x_{\g,\d})$ at a solution $x_{\g,\d}$ of the penalized problem.
     For this, let us define
     \[f^*=\min_{z\in X} f(z).\]
    We have the following result for these optimal values.

\begin{prop}\label{prop:nonstrong}
   Let $\d^0$ be a given accuracy parameter. Let Assumption~\ref{assum:constrSet} hold, and
   assume that $f$ is convex and has bounded level sets.
   Let the parameters $\gamma$ and $\d$ be chosen such that
   \[0<\d < \min\left\{\e,\frac{16\a_{\min}^2}{\b^2 m^2}\right\},\qquad
    \Gamma \le\g\le \frac{4\a_{\min}\delta^0}{\de}, \]
    with
    \[\Gamma
    =\max\left\{L\left(\frac{1}{m\beta} - \frac{\sqrt{\de}}{4\a_{\min}}\right)^{-1},\,4mL\a_{\max}\left( \frac{1}{\sqrt{\d}} + \frac{\b m}{\a_{\min}}\right)\right\},\]
   where $\e$ is the constant from Assumption~\ref{assum:constrSet},
   $\beta$ is the constant from Hoffman's bound (see~Lemma~\ref{lem:hof}), the scalars
   $\a_{\min}$ and $\a_{\max}$ are defined in~\eqref{eq:aminmax},  while
   $L=L(c,\e)$ is the bound on
   the subgradient norms as defined in~\eqref{eq:Lipc} with $c> 4 \a_{\min}\delta^0$.
   Then, we have
   \[0\le f(x_{\g\d}^*) - f^* \le \de^0\qquad\hbox{for any $x_{\g\d}^*\in X_{\g\d}^*$.}\]
 \end{prop}
   \tat{\begin{remark}\label{rem:param3}
 Note that the conditions for $\d$ and $\g$ in Proposition~\ref{prop:nonstrong} are not contradictive. Indeed, $\Gamma\le\frac{4 \a_{\min}\delta^0}{\d}$ for $\delta$ such that
 \begin{align}\label{eq:param3}
   \sqrt{\d} \le \min\left\{\frac{\sqrt{\tilde{c}^2 + \frac{64\a^2_{\min}L\tilde{c}}{m\b}}-\tilde{c}}{8\a_{\min}L}, \frac{\sqrt{\a^2_{\min}+\frac{\tilde{c}\b\a_{\min}}{L\a_{\max}}}-\a_{\min}}{2m\b}\right\}, \mbox{ where }\tilde{c} = 4 \a_{\min}\delta^0.
 \end{align}
 The inequality $\Gamma\le\frac{2\mu_f \a_{\min}\delta^0}{\d}$ implies existence of $\g$ such that $\Gamma\le\g\le\frac{4\a_{\min}\delta^0}{\d}$ holds for any choice of $\d^0$.
 \end{remark}}
 \begin{proof}
  By the assumption that $f$ has bounded level sets, the solution sets $X^*$ and $X^*_{\g\d}$,
   for any $\d,\g\ge0$, are nonempty.
  In view of relations $c> 4 \a_{\min}\delta^0$ and $\g\le \frac{4\a_{\min}\delta^0}{\de}$, it follows that
  $\g\d\le c$. Hence, all the conditions of Proposition~\ref{prop:penSol} are satisfied.
  By Proposition~\ref{prop:penSol}, the solutions of the penalized problem are feasible,
   i.e., $X^*_{\g\d}\subseteq X$, implying that
   \be
   0\le f(x^*_{\g\d}) - f^*.\ee

  Now, let $x^*_{\g\d}\in X^*_{\g\d}$ and $x^*\in X^*$ be arbitrary solutions, and consider the difference
  $F_{\g\d}(x^*_{\g\d}) - F_{\g\d}(x^*).$
  By the definition of the functions $F_{\g\d}$ we have
  \[F_{\g\d}(x^*_{\g\d}) - F_{\g\d}(x^*)
  =f(x^*_{\g\d}) - f(x^*)+ \frac{\gamma}{m}\sum_{i=1}^{m}\left(h_{\d}(x^*_{\g\d};a_i,b_i) - h_{\d}(x^*;a_i,b_i)\right).\]
  Since  $x^*_{\g\d}\in X^*_{\g\d}$, it follows that $F_{\g\d}(x^*_{\g\d}) - F_{\g\d}(x^*)\le0$,
  thus implying that
  \[f(x^*_{\g\d}) - f(x^*)\le
  \frac{\gamma}{m}\sum_{i=1}^{m}\left(h_{\d}(x^*;a_i,b_i) - h_{\d}(x^*_{\g\d};a_i,b_i)\right).\]
  The functions $h_{\d}(\cdot;a_i,b_i)$ are nonnegative, so it follows that
  \[f(x^*_{\g\d}) - f(x^*)\le
  \frac{\gamma}{m}\sum_{i=1}^{m}h_{\d}(x^*;a_i,b_i).\]
  In view of the maximum penalty over feasible region (cf.~\eqref{eq:hfunineq1}) and since $x^*\in X$,
  we have that $h_{\d}(x^*;a_i,b_i)\le \frac{\d}{4\|a_i\|}$ for all $i=1,\ldots,m$. Therefore,
  \[f(x^*_{\g\d}) - f(x^*)\le \frac{\g\d}{4\a_{\min}}.\]
  By the condition $\g\le \frac{4 \a_{\min}\delta^0}{\d}$, it follows that
  \[f(x^*_{\g\d}) - f(x^*)\le \d^0.\]
  \end{proof}

\tat{Propositions~\ref{prop:penSol}-\ref{prop:nonstrong} above connect the solution set of the penalized problem~\eqref{eq:pen-problem} with the solution set of the original problem~\eqref{eq:problem}. They demonstrate that under an appropriate choice of the penalty parameters $\g$ and $\d$ every solution of the penalized problem is feasible and its distance to the solution of~\eqref{eq:problem} does not exceed a predefined constant $\d^0$. This choice depends in particular on the number of constraints $m$, $\a_{\min}$, $\a_{\max}$, Hoffman's constant $\b$, the upper bound of the gradients $L$ in some suitable region.  It is worth noting that a large number of constraints causes a large penalty constant $\g$ and a small smoothness parameter $\d$, what in its turn can lead to an ill-conditioned penalized problem. Moreover, the constants $\b$ and $L$ might be difficult to estimate for a given optimization problem~\eqref{eq:problem}. In the next section we consider a fast incremental method to find a solution of the penalized problem~\eqref{eq:pen-problem} with constant $\g$ and $\d$. According to Propositions~\ref{prop:penSol}-\ref{prop:nonstrong}, this solution will provide a feasible but an inexact solution for the original problem. After addressing this approach, in Section~\ref{sec:varpar}  we will turn our attention to time-varying penalty parameters which do not require an apriori knowledge about $\b$ and $L$. }

\section{Applying SAGA to Penalized Problem}\label{sec:incr}
In this section we formulate fast incremental methods
to find a solution of the penalized problem~\eqref{eq:pen-problem}. Moreover, with the results in Proposition~\ref{prop:strong} and Proposition~\ref{prop:nonstrong} in place we present convergence results for the original optimization problem~\eqref{eq:problem}.

Recently, many algorithms have been proposed to incrementally solve the following optimization problem of minimizing the average sum of functions:
\begin{align}\label{eq:genpr}
 \min_{x\in\R^n} G(x),\qquad G(x) = \frac 1N\sum_{i=1}^N g_i(x).
\end{align}
Among these algorithms are, for example, SAG, SAGA, and SVRG~\cite{saga, finito, svrg},
which leverage the idea to randomly sample the full gradient by processing only one function per iteration
in a way to reduce the variance in the gradient estimation.
Under the assumption of Lipschitz continuous gradients $\nabla g_i$, these algorithms possess the same asymptotic convergence rate to an optimal solution as the standard full gradient method requiring
the full sum of the gradients $\nabla g_i$ at each iteration.
More precisely, given an optimal choice of step size parameters, the aforementioned incremental methods  approach an optimal solution with the convergence rate $O(q^t)$, $q\in(0,1)$, in the case of strongly convex function $G$, and the convergence rate $O(1/t)$ in the case of non-strongly convex function $G$.

As an example of a fast incremental method, we will consider
the SAGA algorithm\footnote{The SAGA method in~\cite{saga} is formulated for a composite objective function
$G(x) = \frac 1N\sum_{i=1}^N g_i(x)+h(x)$, where  the  proximal  operator associated with the convex function
$h$ is  easy  to  evaluate. However, in our setting, it suffices to consider the case $h(x)=0$.}.
The algorithm is summarized as follows.
\begin{algorithm}
\caption{SAGA Algorithm}\label{eq:saga}
0. Let $x^0\in\R^n$ and $\nabla g_i(\phi_i^0)$ with $\phi_i^0 = x^0$, $i=1,\ldots,N$, be known.
\smallskip

1. Pick an index  $j$ uniformly at random.
\smallskip

2. $\phi_j^{t+1} = x^t$  and store $\nabla g_j(\phi_j^{t+1})$.
\smallskip

3.  $x^{t+1} = x^t - \alpha\left[\nabla g_j(\phi_j^{t+1}) - \nabla g_j(\phi_j^{t}) + \frac 1N\sum_{i=1}^N \nabla g_i(\phi_j^{t})\right]$.
\end{algorithm}

The main result for the SAGA algorithm is formulated in the following theorem,
which is adapted from~\cite{saga}.

\begin{thrm}\label{th:saga}(\cite{saga})
\begin{itemize}
\item[(a)]
 Let the functions $g_i$, $i=1,\ldots, N$, be strongly convex with a parameter $\mu>0$ and have Lipschitz continuous gradients with a constant $L_g>0$. Let $x_g^*$ be the solution of the problem~\eqref{eq:genpr}.
 Then, if the step size $\alpha = \frac{1}{2(\mu N + L_g)}$ is chosen in SAGA algorithm, then the following convergence rate result holds:
 $$\mathbb E\|x^t-x_g^*\|^2\le O\left(q^t\right),\qquad q=1-\frac{\mu}{2(\mu N+L_g)}.$$
\item[(b)]
 Let the functions $g_i$, $i=1,\ldots, N$, be non-strongly convex and have Lipschitz continuous gradients with a constant $L_g>0$. Let $G^*$ be the optimal value of the problem~\eqref{eq:genpr} and $\bar x^t = \frac 1t\sum_{k=1}^t x^k$.
 Then, if the step size $\alpha = \frac{1}{3 L_g}$ is chosen, then the following convergence rate result is valid:
 $$\mathbb E[G(\bar x^t)]-G^*\le O\left(\frac{4N}{t}\right).$$\
 \end{itemize}
\end{thrm}

By applying Algorithm~\ref{eq:saga} to the penalized problem~\eqref{eq:pen-problem} under our consideration, namely by taking $g_i(x)=f(x)+\gamma h_\d\left(x; a_i, b_i\right)$, we get the following incremental algorithm to find its solution.

\begin{algorithm}
\caption{SAGA-based Fast Incremental Method for Solving Penalized Problem}\label{eq:saga_pen}
0. Let $x^0\in\R^n$ and $\nabla f(\phi_i^0) + \gamma \nabla h_\d\left(\phi_i^0; a_i, b_i\right)$ with $\phi_i^0 = x^0$, $i=1,\ldots,m$, be known.
\smallskip

1. Pick an index  $j$ uniformly at random.
\smallskip

2. $\phi_j^{t+1} = x^t$ and store $\nabla g_j(\phi_j^{t+1})$.
\smallskip

3. $x^{t+1} = x^t - \alpha[\nabla f(\phi_j^{t+1}) + \gamma \nabla h_\d\left(\phi_j^{t+1}; a_j, b_j\right) - \nabla f(\phi_j^t) - \gamma \nabla h_\d\left(\phi_j^t; a_j, b_j\right)\\  \mbox{\qquad\qquad\qquad\qquad\qquad\qquad\qquad\qquad\qquad}+  \frac 1m\sum_{i=1}^m \left(\nabla f(\phi_i^t) + \gamma \nabla h_\d\left(\phi_i^t; a_i, b_i\right)\right)].     $
\end{algorithm}

In terms of the original optimization problem~\eqref{eq:problem} the following result holds, as a
direct consequence
of Theorem~\ref{th:saga}, and Proposition~\ref{prop:strong} and Proposition~\ref{prop:nonstrong}.

\begin{thrm}\label{th:saga_pen}
Let Assumption~\ref{assum:constrSet} hold.
\begin{itemize}
\item[(a)]
Let the function $f$ be strongly convex with a parameter $\mu_f>0$ and have Lipschitz continuous gradients with a constant $L_f>0$. Let $x^*$ be the solution of the problem \eqref{eq:problem}.
Assume that an accuracy level $\delta^0$ is given, the penalty parameters $\gamma$ and $\delta$ are chosen to satisfy the conditions of Proposition~\ref{prop:strong}, and the step size
 $\alpha = \frac{1}{2(\mu_f m+ L_f + \frac{\g\a_{\max}}{2\d})}$ is selected.
 Then, the following convergence rate result is valid for the iterates of Algorithm~\ref{eq:saga_pen}:
 \begin{align*}
  \mathbb E[\|x^t-\Pi_X[x^t]\|^2]\le O\left(q_{\g}^t\right),\quad
 \mathbb E[\|x^t-x^*\|]^2&\le O\left(q_{\g}^t\right) + 2\delta^0,\cr
 &q_{\g}=1-\frac{\mu_f}{2(\mu_f m +L_f + \frac{\g\a_{\max}}{2\d})}.
 \end{align*}
\item[(b)]
 Let the function $f$ be convex, have bounded level sets and have Lipschitz continuous gradients
 with a constant $L_f>0$. Let $f^*$ be the optimal value problem~\eqref{eq:problem}.
 Suppose that a desired accuracy level $\delta^0$ is given, the penalty parameters $\gamma$ and $\delta$
 are chosen to satisfy the conditions of Proposition~\ref{prop:nonstrong}, and the step size $\a$ is given by
 $\alpha = \frac{1}{3(L_f+\frac{\g\a_{\max}}{2\d})}$.
 Then, for the averages $\bar x^t = \frac{1}{t}\sum_{k=1}^t x^k$ of the iterates $x^k$ generated  by
 Algorithm~\ref{eq:saga_pen} the following holds:
 \[E [f(\bar x^t)] - f^*\le O\left(\frac{4m}{t}\right)+2\delta^0\qquad\hbox{for all }t,\]
 and for any $ \tilde \e>0$ there exists $T>0$ such that for all
 $t>T$,
 $$-\g\tilde\e\le\mathbb E [f(\bar x^t)] - f^*.$$
 \end{itemize}
\end{thrm}
\begin{proof}
We apply Theorem~\ref{th:saga} to the problem \eqref{eq:pen-problem} with the objective function
 \[F_{\g\d} = \frac 1m\sum_{i=1}^m g_i(x),\]
 where $g_i(x)=f(x)+\gamma h_\d\left(x; a_i, b_i\right)$, $i=1,\ldots,m$.
 Recall that, if the function $f$ is strongly convex with a constant $\mu_f>0$, then the functions $g_i(x)$ are strongly convex with the same constant
 $\mu_f$. Moreover, the gradients of the functions $g_i(x)$ are Lipschitz continuous with the constant $L_f + \frac{\g\a_{\max}}{2\d}$,
since each penalty function $h_\d\left(x; a_i, b_i\right)$, $i=1,\ldots,m$,
 has the Lipschitz continuous gradient with the constant $\frac{\|a_i\|}{2\d}$ (see~\eqref{eq:Lipc-gradh}).

 To obtain the result in part (a), let us notice that, according to Proposition~\ref{prop:penSol}
 we have $x^*_{\g\d}\in X$. Hence,
 $$\mathbb E[\|x^t-\Pi_X[x^t]\|^2]\le \mathbb E[\|x^t-x^*_{\g\d}\|^2].$$
 Thus, due to Theorem~\ref{th:saga}(a), Proposition~\ref{prop:strong}, and
 the inequality $\|a+b\|^2\le 2\|a\|^2 + 2\|b\|^2$, which is valid for any $a,b\in\R^n$, we  conclude that
 $$\mathbb E[\|x^t-\Pi_X[x^t]\|^2]\le O\left(q_{\g}^t\right),\qquad \mathbb E[\|x^t-x^*\|^2]\le O\left(q_{\g}^t\right) + 2\delta^0,$$
 where $q_{\g}=1-\frac{\mu_f}{2(\mu_f m +L_f + \frac{\g}{2\d})}$.

To prove part (b),
 we consider $\mathbb E [F_{\g\d}(\bar x^t)] - F^*_{\g\d}$, for which, according to Theorem~\ref{th:saga}(b), we have
 \begin{align}\label{eq:ineq}
   0\le \mathbb E [F_{\g\d}(\bar x^t)] - F^*_{\g\d} \le  O\left(\frac{4m}{t}\right).
 \end{align}
By using the definition of the penalty function  $F_{\g\d}$, for any $x^*_{\g\d}\in X^*_{\g\d}$ we can write
\[  \mathbb E [F_{\g\d}(\bar x^t)] - F^*_{\g\d}
=\mathbb E [f(\bar x^t)] - f(x^*_{\g\d}) + \frac{\g}{m}\sum_{i=1}^{m}\left\{\mathbb E [h_{\d}(\bar x^t; a_i, b_i)] -  h_{\d}(x^*_{\g\d}; a_i, b_i)\right\}.\]
From the preceding relation, using the fact that the functions $h_{\d}(\cdot; a_i, b_i)$ are nonnegative
and using~\eqref{eq:ineq}, we obtain
\[\mathbb E [f(\bar x^t)] - f(x^*_{\g\d})
\le O\left(\frac{4m}{t}\right)
+ \frac{\g}{m}\sum_{i=1}^{m} h_{\d}(x^*_{\g\d}; a_i, b_i).\]
By adding and subtracting $f^*$ and re-arranging the terms, we further obtain
\[\mathbb E [f(\bar x^t)] - f^*
\le O\left(\frac{4m}{t}\right) +f(x^*_{\g\d})-f^*+ \frac{\g}{m}\sum_{i=1}^{m} h_{\d}(x^*_{\g\d}; a_i, b_i).\]
By Proposition~\ref{prop:nonstrong}, we have $f(x^*_{\g\d}) - f^*\le\d^0$, implying that
\[\mathbb E [f(\bar x^t)] - f^*
\le O\left(\frac{4m}{t}\right) +\d^0+ \frac{\g}{m}\sum_{i=1}^{m} h_{\d}(x^*_{\g\d}; a_i, b_i).\]
By Proposition~\ref{prop:penSol}, the point $x_{\g\d}^*$ is feasible for the original problem, so that we have
$h_{\d}(x^*_{\g\d}; a_i, b_i)\le \frac{\d}{4\|a_i\|}$ for all $i=1,\ldots,m$ (see~\eqref{eq:hfunineq1}). Therefore,
for all $t$ we have
\[\mathbb E [f(\bar x^t)] - f^*
\le O\left(\frac{4m}{t}\right) +\d^0+ \frac{\g\d}{4\a_{\min}}\le O\left(\frac{4m}{t}\right) +2\d^0,\]
where the last inequality follows in view of the condition $\g\le \frac{4\a_{\min}\d^0} {\d}$
of Proposition~\ref{prop:nonstrong}.

Next, we provide a lower bound on $\mathbb E [f(\bar x^t)] - f^*$.
We write
\[\mathbb E [f(\bar x^t)] - f^*=\mathbb E [f(\bar x^t)] -  f(x^*_{\g\d}) + f(x^*_{\g\d})-f^*
\ge \mathbb E [f(\bar x^t)] -  f(x^*_{\g\d}),\]
where $x^*_{\g\d}$ is an arbitrary solution of the penalized problem, i.e., $F(x^*_{\g\d})= F^*_{\g\d}$,
and  the last inequality is obtained by using
$f(x^*_{\g\d})-f^*\ge0$ (see Proposition~\ref{prop:nonstrong}).
By using the convexity of $f$, we further have
\[[\mathbb E [f(\bar x^t)] - f^*\ge f( \mathbb E [\bar x^t]) -  f(x^*_{\g\d}).\]
By the definition of the penalty function, we have
\[f(\mathbb E [\bar x^t]) - f(x^*_{\g\d})=F_{\g\d}(\mathbb E [\bar x^t]) - F^*_{\g\d}
- \frac{\g}{m}\sum_{i=1}^{m}\left\{h_{\d}(\mathbb E [\bar x^t]; a_i, b_i) -  h_{\d}(x^*_{\g\d}; a_i, b_i)\right\}.\]
Hence, for any $x^*_{\g\d}\in X^*_{\g\d}$,
\begin{align*}
\mathbb E [f(\bar x^t)] - f^*
&\ge F_{\g\d}(\mathbb E [\bar x^t]) - F^*_{\g\d}
+ \frac{\g}{m}\sum_{i=1}^{m}\left\{h_{\d}(x^*_{\g\d}; a_i, b_i) - h_{\d}(\mathbb{E} [\bar x^t]; a_i, b_i)\right\}\cr
&\ge
\frac{\g}{m}\sum_{i=1}^{m}\left\{ h_{\d}(x^*_{\g\d}; a_i, b_i) - h_{\d}(\mathbb{E} [\bar x^t]; a_i, b_i)\right\},
\end{align*}
where the last inequality follows by $F_{\g\d}(\mathbb E [\bar x^t]) - F^*_{\g\d}\ge 0$ since $x^*_{\g\d}$ is a minimizer of $F_{\g\d}$. The function $h_\d(\cdot;a,b)$ has bounded gradient norms by 1 (see~\eqref{eq:sfun}--\eqref{eq:gradh}), implying that for all $i=1,\ldots,m$ and for all $x^*_{\g\d}\in X^*_{\g\d}$,
\[h_{\d}(x^*_{\g\d}; a_i, b_i) - h_{\d}(\mathbb{E} [\bar x^t]; a_i, b_i)\ge -\|x^*_{\g\d}-\mathbb{E} [\bar x^t]\|,\]
By choosing a particular solution $\Pi_{X^*_{\g\d}}[\mathbb{E} [\bar x^t]],$
we have
\begin{align}\label{eq:p1}
\mathbb E [f(\bar x^t)] - f^*
\ge-\g\|\Pi_{X^*_{\g\d}} [ \mathbb{E} [\bar x^t]] -\mathbb{E} [\bar x^t]\|.
\end{align}
 According to Theorem~\ref{th:saga}(b), for Algorithm~\ref{eq:saga_pen} there holds
 \[\lim_{t\to\infty}\mathbb E F_{\g\d}(\bar x^t) = F^*_{\g\d}.\]
 By the convexity of the function $F_{\g\d}$ and the fact that $F^*_{\g\d}$ is its unconstrained minimum value,
 it follows that
 \[\lim_{t\to\infty}F_{\g\d}(\mathbb E \bar x^t) = F^*_{\g\d}.\]
 Thus, any limit point of the sequence $\{\mathbb E \bar x^t\}$ belongs to the set of minimizers $X^*_{\g\d}$
 of the function $F_{\g\d}$.
Hence, for any given $\tilde \epsilon>0$ there exists $T>0$ such that for all $t>T$,
 \begin{align}\label{eq:lim}
  \|\mathbb E [\bar x^t] - \Pi_{X^*_{\g\d}}[\mathbb{E}[ \bar x^t]]\|\le \tilde \epsilon.
 \end{align}
 The result follows from~\eqref{eq:p1} and~\eqref{eq:lim}.
\end{proof}

We emphasize that Algorithm~\ref{eq:saga_pen} presented above is just an example of fast incremental methods which use the Lipschitz gradient property of the objective function and are, thus, applicable to the penalized optimization problem~\eqref{eq:pen-problem}. Other methods with potentially better rate dependence on the problem's parameters include \cite{Katusha, finito, svrg, accSAGA, scda}. All these algorithms guarantee a fast convergence to a feasible point lying within some $\delta^0$-neighborhood of an optimal solution
for a predefined accuracy parameter $\delta^0>0$. \tat{Unfortunately (and naturally) the following issues may arise while applying these incremental methods to the penalized problem~\eqref{eq:pen-problem} in the case of large $m$.
As it has been mentioned above, a large number of constraints leads to a small $\delta$. It implies a small step size and a low convergence rate (see Theorem~\ref{th:saga_pen}).
Moreover, a small smoothing parameter $\d$ makes the penalized problem~\eqref{eq:pen-problem} closer to an ill-conditioned one. In the case of strongly convex $f$, we rectify this issue in Section~\ref{sec:varpar}. In Section~\ref{sec:varpar} we propose an alternative approach to solve the penalized problem. This approach is based on time-dependent penalty parameters and guarantees convergence of the full gradient-based method to the exact solution of the original problem~\eqref{eq:problem}.}

\subsection{Simulation Results}\label{sec:numer}
\tat{Before addressing time-varying penalty parameters, we test the theoretic results presented above. We consider the following regression problem
\begin{align}\label{eq:implement}
 &\min_{x\in\R^n} \|\Phi x- x^0\|^2\cr
 &\mbox{s.t. } Ax\le b,
\end{align}
where $\Phi\in \R^{l\times n}$ is a feature matrix, $x^0\in\R^l$ is a reference point, and $A\in \R^{m\times n}$, $b\in\R^m$ define the set of linear inequality constraints.
Here the coordinates of $x^0\in\R^l$ are chosen at random from the normal distribution with the mean value 0 and variance 10. The entries of $\Phi$ were independently
generated according to a uniform distribution in $[-1,1]$, whereas the elements of $A$ and $b$ are chosen at random from the normal distribution with the mean value 0 and variance 100\footnote{The entries of $A$ and $b$ were adapted until Assumption~\ref{assum:constrSet} is satisfied for the reference point $\bar x = 0$ and $\e=0.01$.}.
For implementation we set $l=n=30$.

Figure~\ref{fig:effG} demonstrates the distance $\|x^t-x^*\|$, where $x^*$ is the solution of \eqref{eq:implement}, estimated for the iterations of the standard full gradient procedure averaged over 100 runs, that was implemented for the penalized problem in the case $m=500$ with the fixed $\d = 0.001$ and different parameters $\g$ starting by $\g=100m^2$. As we can see, the increase of $\g$ implies a faster approach of the solution. However, without adapting the smoothness parameter $\d$ the output after 1000 iterations may stay infeasible. In the current implementation the choice $\g=100m^2$, $\g=2*100m^2$,  $\g=5*100m^2$, $\g=10*100m^2$ resulted in a feasible $x^{100}$ in 100\% of the implementations, whereas the choice $\g=20*100m^2$ led to an infeasible  $x^{100}$ in 33\% of the runs.

The run of the full gradient method (FullGrad) and two other algorithms, namely the SAGA procedure for solving the problem based on the penalized function approach (PA/SAGA) from Algorithm~\ref{eq:saga_pen} and the random projection algorithm (RandProj) from \cite{Nedich2011}, are presented on Figures~\ref{fig:m500}-\ref{fig:m1000_av} for the number of inequality constraints $m=500$ and $m=1000$. Note that Figures~\ref{fig:m500}-\ref{fig:m1000} demonstrate the single run of the procedures, whereas Figure~\ref{fig:m1000_av} shows the result averaged over 100 runs. The step-size parameters for FullGrad and PA/SAGA as well as the penalty parameters in PA/SAGA were tuned. The step-size parameter for RandProj was chosen as $\frac{1}{t}$ according to the theoretic results provided in~\cite{Nedich2011}.  As we can see, during the first $1000$ iterations the SAGA-based algorithm outperforms the random projection procedure by decreasing the relative error ${\|x^t-x^*\|}$ faster. Moreover, the termination state $x^{1000}$ in the case of the averaged runs (Figure~\ref{fig:m1000_av}) occurs to be non-feasible in around $26\%$ of implementations, whereas for the tuned parameters $\delta$ and $\gamma$ all runs of PA/SAGA terminate at a feasible point $x^{1000}$.

}

\begin{figure}[!h]

\begin{minipage}[b]{.5\textwidth}
\centering
\psfrag{0}[c][b]{\tiny$0$}
\psfrag{50}[c][b]{\tiny $50$}
\psfrag{150}[c][b]{\tiny$150$}
\psfrag{250}[c][b]{\tiny$250$}
\psfrag{350}[c][b]{\tiny$350$}
\psfrag{450}[c][b]{\tiny$450$}
\psfrag{100}[c][b]{\tiny$100$}
\psfrag{200}[c][b]{\tiny$200$}
\psfrag{300}[c][b]{\tiny$300$}
\psfrag{400}[c][b]{\tiny$400$}
\psfrag{500}[c][b]{\tiny$500$}
\psfrag{600}[c][b]{\tiny$600$}
\psfrag{700}[c][b]{\tiny$700$}
\psfrag{800}[c][b]{\tiny$800$}
\psfrag{900}[c][b]{\tiny$900$}
\psfrag{1000}[c][b]{\tiny$1000$}
\psfrag{5}[c][l]{\tiny$5$}
\psfrag{15}[c][l]{\tiny $15$}
\psfrag{25}[c][l]{\tiny$25$}
\psfrag{20}[c][l]{\tiny$20$}
\psfrag{40}[c][l]{\tiny $40$}
\psfrag{45}[c][l]{\tiny $45$}
\psfrag{35}[c][l]{\tiny$35$}
\psfrag{30}[c][l]{\tiny$30$}
\psfrag{time}[c][b]{\tiny time}
\psfrag{Relative Error}[b][r]{\tiny \vspace{-5cm}}
\begin{overpic}[width=1\textwidth]{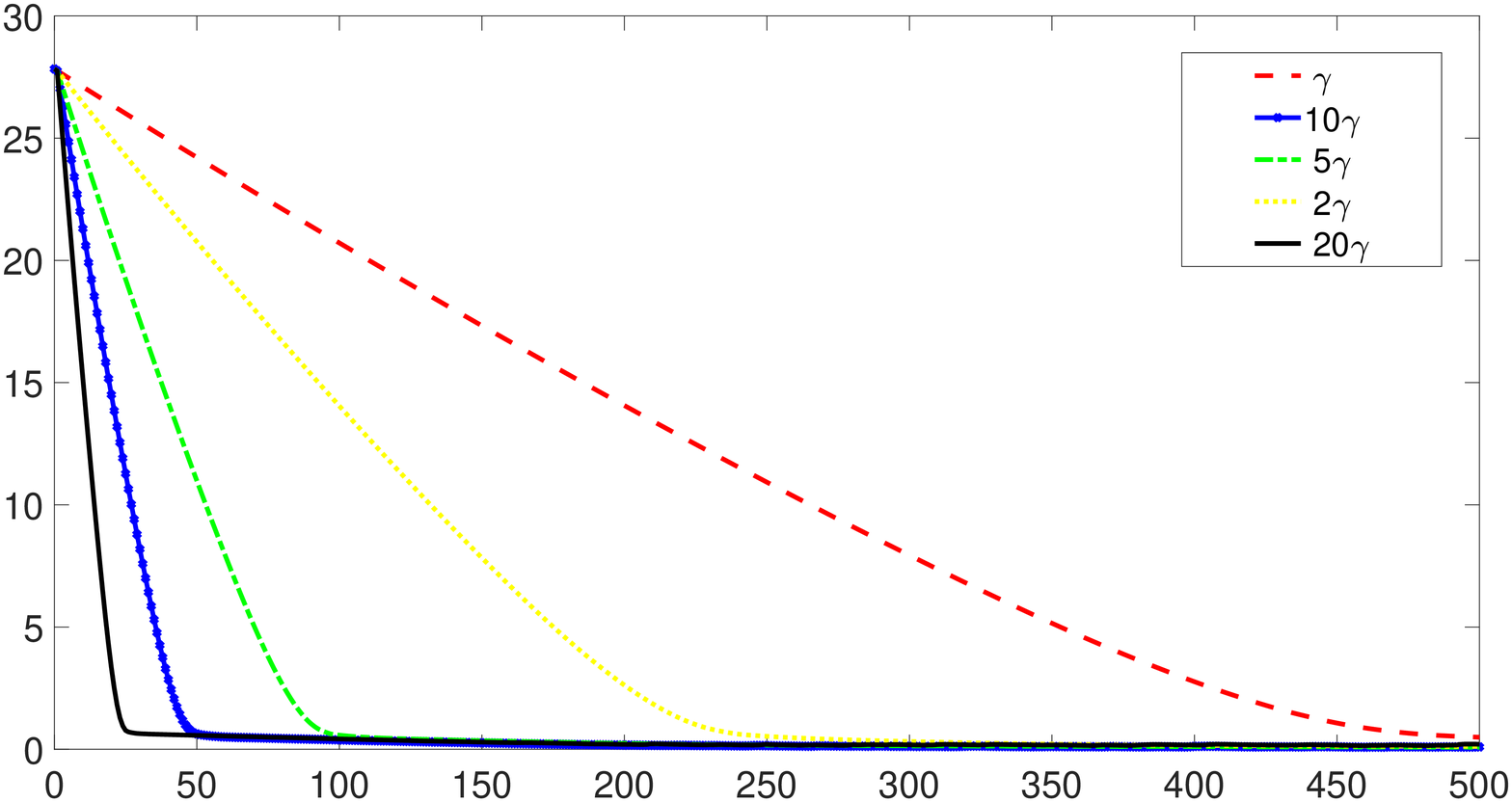}
\put(6,18){\rotatebox{90}{\tiny{Relative Error}}}
\end{overpic}
\caption{Effect of $\g$.}
\label{fig:effG}
\end{minipage}
\hfill
\begin{minipage}[b]{.5\textwidth}
\centering
\psfrag{0}[c][b]{\tiny$0$}
\psfrag{100}[c][b]{\tiny$100$}
\psfrag{200}[c][b]{\tiny$200$}
\psfrag{300}[c][b]{\tiny$300$}
\psfrag{400}[c][b]{\tiny$400$}
\psfrag{500}[c][b]{\tiny$500$}
\psfrag{600}[c][b]{\tiny$600$}
\psfrag{700}[c][b]{\tiny$700$}
\psfrag{800}[c][b]{\tiny$800$}
\psfrag{900}[c][b]{\tiny$900$}
\psfrag{1000}[c][b]{\tiny$1000$}
\psfrag{5}[c][l]{\tiny$5$}
\psfrag{10}[c][l]{\tiny $10$}
\psfrag{15}[c][l]{\tiny $15$}
\psfrag{25}[c][l]{\tiny$25$}
\psfrag{20}[c][l]{\tiny$20$}
\psfrag{30}[c][l]{\tiny$30$}
\psfrag{time}[c][b]{\tiny time}
\psfrag{Relative Error}[b][r]{\tiny \vspace{-5cm}}
\begin{overpic}[width=1\textwidth]{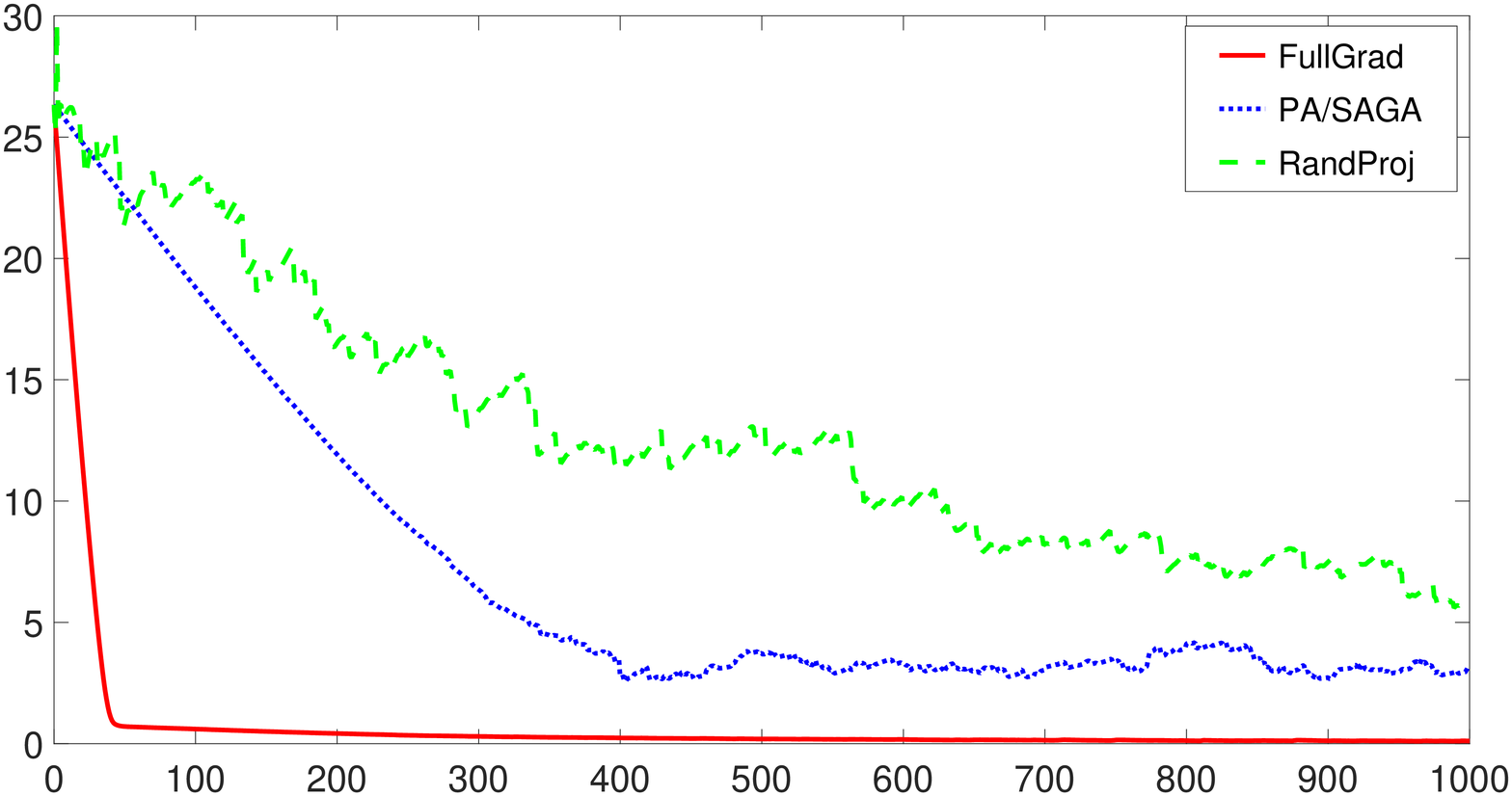}
\put(6,18){\rotatebox{90}{\tiny{Relative Error}}}
\end{overpic}
\caption{PA/SAGA vs RandProj, $m=500$.}
\label{fig:m500}
\end{minipage}


\vspace{2ex}


\noindent\begin{minipage}[b]{.5\textwidth}
\centering
\psfrag{0}[c][b]{\tiny$0$}
\psfrag{100}[c][b]{\tiny$100$}
\psfrag{200}[c][b]{\tiny$200$}
\psfrag{300}[c][b]{\tiny$300$}
\psfrag{400}[c][b]{\tiny$400$}
\psfrag{500}[c][b]{\tiny$500$}
\psfrag{600}[c][b]{\tiny$600$}
\psfrag{700}[c][b]{\tiny$700$}
\psfrag{800}[c][b]{\tiny$800$}
\psfrag{900}[c][b]{\tiny$900$}
\psfrag{1000}[c][b]{\tiny$1000$}
\psfrag{5}[c][l]{\tiny$5$}
\psfrag{10}[c][l]{\tiny $10$}
\psfrag{15}[c][l]{\tiny $15$}
\psfrag{25}[c][l]{\tiny$25$}
\psfrag{20}[c][l]{\tiny$20$}
\psfrag{40}[c][l]{\tiny $40$}
\psfrag{45}[c][l]{\tiny $45$}
\psfrag{35}[c][l]{\tiny$35$}
\psfrag{30}[c][l]{\tiny$30$}
\psfrag{time}[c][b]{\tiny time}
\psfrag{Relative Error}[b][r]{\tiny \vspace{-5cm}}
\begin{overpic}[width=1\textwidth]{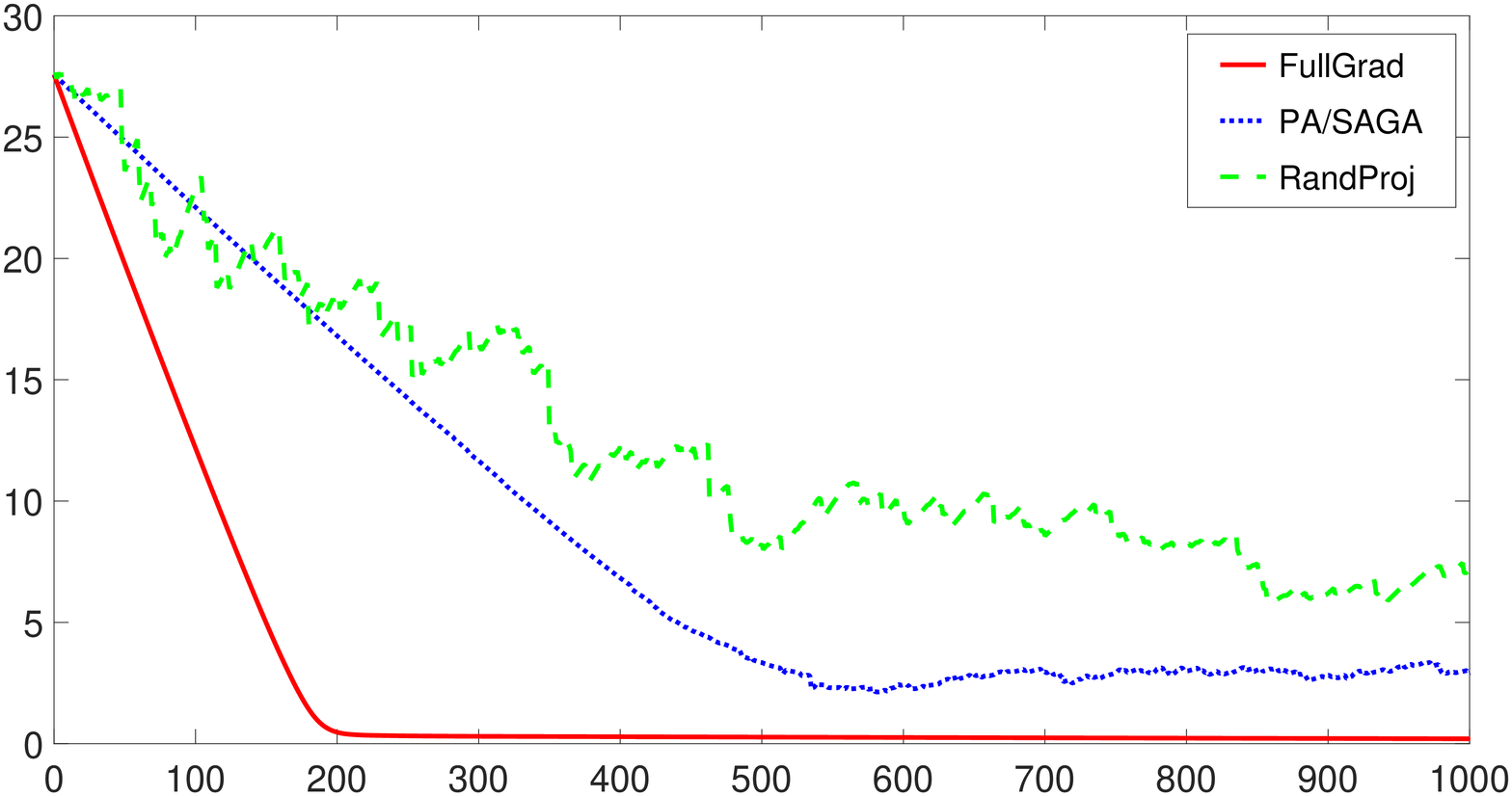}
\put(6,18){\rotatebox{90}{\tiny{Relative Error}}}
\end{overpic}
\caption{PA/SAGA vs RandProj, $m=1000$.}
\label{fig:m1000}
\end{minipage}
\hfill
\begin{minipage}[b]{.5\textwidth}
\centering
\psfrag{0}[c][b]{\tiny$0$}
\psfrag{100}[c][b]{\tiny$100$}
\psfrag{200}[c][b]{\tiny$200$}
\psfrag{300}[c][b]{\tiny$300$}
\psfrag{400}[c][b]{\tiny$400$}
\psfrag{500}[c][b]{\tiny$500$}
\psfrag{600}[c][b]{\tiny$600$}
\psfrag{700}[c][b]{\tiny$700$}
\psfrag{800}[c][b]{\tiny$800$}
\psfrag{900}[c][b]{\tiny$900$}
\psfrag{1000}[c][b]{\tiny$1000$}
\psfrag{5}[c][l]{\tiny$5$}
\psfrag{10}[c][l]{\tiny $10$}
\psfrag{15}[c][l]{\tiny $15$}
\psfrag{25}[c][l]{\tiny$25$}
\psfrag{20}[c][l]{\tiny$20$}
\psfrag{60}[c][l]{\tiny $60$}
\psfrag{50}[c][l]{\tiny $50$}
\psfrag{40}[c][l]{\tiny$40$}
\psfrag{30}[c][l]{\tiny$30$}
\psfrag{time}[c][b]{\tiny time}
\psfrag{Relative Error}[b][r]{\tiny \vspace{-5cm}}
\begin{overpic}[width=1\textwidth]{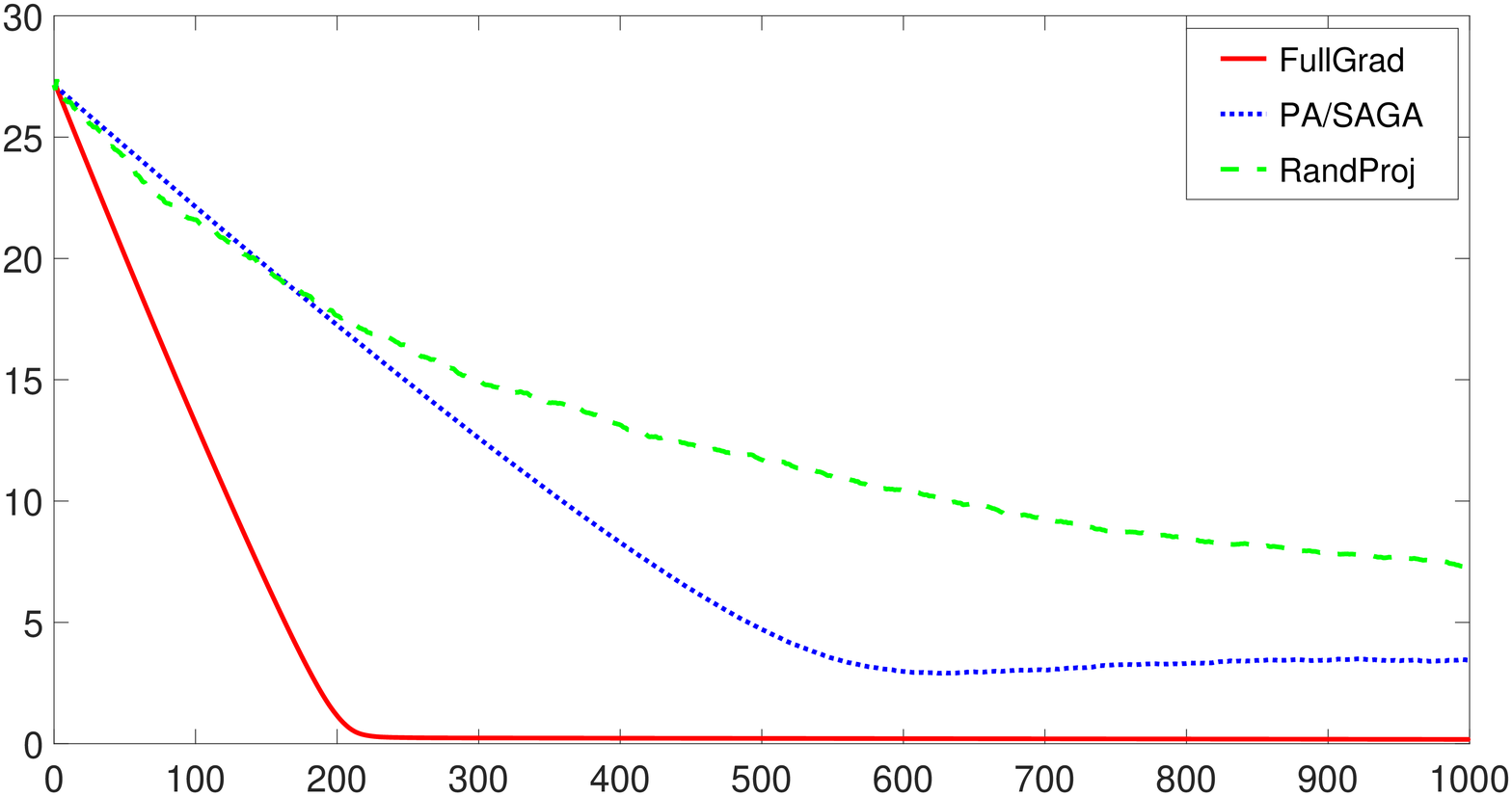}
\put(6,18){\rotatebox{90}{\tiny{Relative Error}}}
\end{overpic}
\caption{av PA/SAGA vs RandProj, $m=1000$.}
\label{fig:m1000_av}
\end{minipage}

\end{figure}

\tat{
\section{Varying Penalty Parameters}\label{sec:varpar}
The results in the preceding section impose
particular conditions on the penalty parameters $\delta$ and $\gamma$
which involve several constants, including
$\e$ from Assumption~\ref{assum:constrSet} and
Hoffman's constant $\b$. These constants may be difficult to obtain for a given problem,
and one may need to resort to an alternative approach, where
the penalty parameters $\delta$ and $\gamma$ are varying. In particular, one may consider
decreasing the values of $\d$ to 0, which we investigate in this section, under the assumption that
the function $f$ is strongly convex and has Lipschitz continuous gradients.
}
\subsection{Behavior of sequence of solutions of penalized problems}
\tat{We consider sequences $\{\delta_k\}$ and $\{\g_k\}$ of positive scalars,
and we denote the corresponding penalized
function $F_{\d_k\g_k}(x)$ simply by $F_k(x)$,
i.e.,
\begin{equation}
F_k(x) = f(x) + \frac{\g_k}{m} \sum_{i=1}^{m} h_k\left(x; a_i, b_i\right),
\end{equation}
where we use $h_k$ to denote the function $h_{\d_k}$.
The corresponding unconstrained optimization problem is
\begin{equation}\label{eq:penprob_k}
\min_{x\in \R^n} F_k(x).
\end{equation}
When $f$ is strongly convex, each of these penalized problems has a unique solution,
denoted by $x_k^*$, and the original problem also has a unique solution $x^*\in X$.
Our next result provides relations for the points $x_k^*$ and
the optimal solution $x^*$ of the original problem.

\begin{lemma}\label{lema-solsbded}
Let $f$ be strongly convex with a constant $\mu>0$.
Assume that for all $k$ the sequences $\{\d_k\}$ and $\{\g_k\}$ are such that $\g_k>0$, $\d_k>0$
and
$\g_k\d_k\le c$ for some positive $c$.
Then, the sequence $\{x_k^*\}$ of solutions to~\eqref{eq:penprob_k}
is contained in the level set
\[\mathcal{L}=\left\{x\in\R^n\mid f(x)\le f(x^*)+\frac{c}{4\a_{\min}}\right\},\]
where $x^*$ is the solution of the original problem and
$\a_{\min}=\min_{1\le i\le m}\|a_i\|$.
In particular, the sequence $\{x_k^*\}$ is bounded.
\end{lemma}
\begin{proof}
By Corollary~\ref{cor:lset}, where we set $\hat x=x^*$,
$\d=\d_k$, $\g=\g_k$,
and
$t_{\g\d}(\hat x)$ is replaced with
$f(x^*)+\frac{\g_k\d_k}{4\a_{\min}}$, we obtain
that
$f(x_k^*)\le f(x^*)+\frac{\g_k\d_k}{4\a_{\min}}$. By using $\g_k\d_k\le c$,
we conclude that $\{x_k\}$ is contained in
the level set $\mathcal{L}=\left\{x\in\R^n\mid f(x)\le f(x^*)+\frac{c}{4\a_{\min}}\right\}$.
Since $f$ is strongly convex, its level sets are bounded, thus implying that
$\{x_k^*\}$ is bounded.
\end{proof}

We next consider a set of conditions on parameters $\d_k$ and $\g_k$ that will ensure
that the sequence $\{x_k^*\}$ converges to $x^*$ as $k\to\infty$.
In what follows, we will use the projections of
the points $x^*_k$ on the feasible set, which
we denote by $p_k$, i.e., $p_k=\Pi_X[x_k^*]$.
Under the assumptions of Lemma~\ref{lema-solsbded},
the sequence$\{x_k^*\}$ is bounded, and so is the sequence
$\{p_k\}$ of the projections of $x_k^*$'s on $X$.
Let $R$ be large enough so that $\{x_k\}\subset \mathbb{B}(0,R)$ and
$\{p_k\}\subset \mathbb{B}(0,R)$, where $\mathbb{B}(0,R)$ denotes the ball centered at the
origin with the radius $R$.
The subgradients of $f(x)$ for $x\in \mathbb{B}(0,R)$ are bounded, and let
$L$ be the maximum norm
of the subgradients of $f(x)$ over $x\in \mathbb{B}(0,R)$,
i.e.,
\begin{equation}\label{eq:sgdbound}
L=\max\{\|s\|\mid s\in\partial f(x), \|x\|\le R\}.
\end{equation}

We have the following lemma.
\begin{lemma}\label{lema-sols}
Let $f$ be strongly convex with a constant $\mu>0$.
Assume that for all $k$ the sequences $\{\d_k\}$ and $\{\g_k\}$ are such that $\g_k>0$, $\d_k>0$
and
$\g_k\d_k\le c$ for some positive $c$. Let $L$ be given by~\eqref{eq:sgdbound}. Then,
for all $k$, we have
\[\frac{\mu}{2}\|x^* -x_k^*\|^2 + \frac{\mu}{2}\|x^*-p_k\|^2
+\left(\frac{\g_k}{4m\b} - L\right) \dist(x_k^*,X)
\le \frac{\g_k\d_k}{4\a_{\min}},\]
where $p_k=\Pi_X[x_k^*]$ for all $k$, and
$\a_{\min}=\min_{1\le i\le m}\|a_i\|$.
\end{lemma}
\begin{proof}
Since each $F_k$ is strongly convex with the constant $\mu>0$,
by the optimality of $x_k^*$ and by the definition of $F_k$, we have
\begin{equation}\label{eq:11}
\frac{\mu}{2}\|x^* -x_k^*\|^2\le 
f(x^*)-f(x_k^*) + \frac{\g_k}{m} \sum_{i=1}^{m} h_k\left(x^*; a_i, b_i\right)
-\frac{\g_k}{m} \sum_{i=1}^{m} h_k\left(x_k^*; a_i, b_i\right).
\end{equation}
By adding and subtracting $f(p_k)$,
we have
\[f(x^*)-f(x_k^*)=f(x^*)-f(p_k) + f(p_k) - f(x_k^*)\le -\frac{\mu}{2}\|x^*-p_k\|^2
+|f(p_k)-f(x_k^*)|,\]
where in the last inequality is obtained using
\[\frac{\mu}{2}\|p_k-x^*\|^2 + \la \nabla f(x^*), p_k-x^*\ra +f(x^*)\le f(p_k),\]
and the fact $\la \nabla f(x^*), p_k-x^*\ra\ge0$, which holds since $p_k$ is feasible
and $x^*$ is the optimal point.
Since the subgradients of $f$ are uniformly bounded at $x^*_k$ and $p_k$,
we have
\[f(x^*)-f(x_k^*)\le -\frac{\mu}{2}\|x^*-p_k\|^2 + L \|p_k - x_k^*\|.\]
Combining the preceding inequality with relation~\eqref{eq:11}, we obtain
\[\frac{\mu}{2}\|x^* -x_k^*\|^2\le -\frac{\mu}{2}\|x^*-p_k\|^2 + L \|p_k - x_k^*\|
+ \frac{\g_k}{m} \sum_{i=1}^{m} h_k\left(x^*; a_i, b_i\right)
-\frac{\g_k}{m} \sum_{i=1}^{m} h_k\left(x_k^*; a_i, b_i\right).\]
Since $x^*$ is feasible, by relation~\eqref{eq:hfunineq1} (where $\d=\d_k$)
we have for all $i=1,\ldots,m$,
\[h_k(x^*; a_i,b_i) \le \frac{\d_k}{4\|a_i\|}\le \frac{\d_k}{4\a_{\min}}.\]
Hence, it follows that
\[\frac{\mu}{2}\|x^* -x_k^*\|^2 + \frac{\mu}{2}\|x^*-p_k\|^2 \le L \|p_k - x_k^*\|
+ \frac{\g_k\d_k}{4\a_{\min}}
-\frac{\g_k}{m} \sum_{i=1}^{m} h_k\left(x_k^*; a_i, b_i\right).\]

By Lemma~\ref{lem:penalty}, where $\d=\d_k$, $a=a_i$, $b=b_i$, and $Y=X_i$,
we have
\[h_k(x_k^*;a_i,b_i)\ge \frac{1}{4}\dist(x_k^*,X_i)
\qquad\hbox{for all $i=1,\ldots,m$}.\]
Therefore,
\[\frac{\mu}{2}\|x^* -x_k^*\|^2 + \frac{\mu}{2}\|x^*-p_k\|^2 \le L \|p_k - x_k^*\|
+  \frac{\g_k\d_k}{4\a_{\min}}
-\frac{\g_k}{4m} \sum_{i=1}^{m} \dist(x_k^*,X_i).\]
By Lemma~\ref{lem:hof}, we have
$\beta \sum_{i=1}^{m} \dist(x, X_i) \ge \dist(x,X)$ for all  $x\in\R^n$,
implying that
\[\frac{\mu}{2}\|x^* -x_k^*\|^2 + \frac{\mu}{2}\|x^*-p_k\|^2 \le L \|p_k - x_k^*\|
+ \frac{\g_k\d_k}{4\a_{\min}}
-\frac{\g_k}{4m\b} \dist(x_k^*,X).\]
Since $p_k$ is the projection of $x_k^*$ on $X$, we have
$\|p_k-x_k^*\|=\dist(x_k^*,X)$ implying that
\[\frac{\mu}{2}\|x^* -x_k^*\|^2 + \frac{\mu}{2}\|x^*-p_k\|^2
\le \left(L -\frac{\g_k}{4m\b}\right) \dist(x_k^*,X)+ \frac{\g_k\d_k}{4\a_{\min}}.\]
\end{proof}

Lemma~\ref{lema-sols} indicates that, when $\g_k\to+\infty$,
for all large enough $k$, we will have $\frac{\g_k}{4m\b}>0$,
implying that
\[\dist(x_k^*,X)\le \frac{\g_k\d_k}{4\a_{\min} \left(\frac{\g_k}{4m\b}-L\right)}
=\frac{m\b\g_k\d_k}{\a_{\min}(\g_k-4m\b L)}
\approx O(\d_k).\]
Thus, if $\d_k\to0$,  the distance of $x_k^*$ to the feasible set $X$ will go to 0
at the rate of  $O(\d_k)$.
Lemma~\ref{lema-sols} also indicates that, for large enough $k,$
\[\|x^* -x_k^*\|^2\le  \frac{\g_k\d_k}{2\mu\a_{\min}}.\]
Thus, if $\g_k\d_k\to0$, then the points $x_k^*$ approach
the optimal solution $x^*$ of the original problem, with the rate of
$O(\g_k\d_k)$.

To summarize, Lemma~\ref{lema-sols} characterizes the behavior of the sequence $\{x_k^*\}$
in terms of the penalty parameters $\{\g_k\}$ and $\{\d_k\}$. It shows that
under conditions $\g_k\to\infty$, $\d_k\to0$ and $\g_k\d_k\to0$, we have
$\|x_k^*-x^*\|\to0$. Based on Lemma~\ref{lema-sols}, one can construct two-loop approach to compute
the optimal point $x^*$ of the original problem,
where for every outer loop $k$, we have an inner loop of iterations to compute $x_k^*$. This, however,
will be quite inefficient. In the next section, we propose a more efficient single-loop algorithm,
where at each iteration $k$ we use the gradient of the function $F_k$.
}

\subsection{A gradient algorithm for solving the original problem}
\tat{The result of Lemma~\ref{lema-sols} is useful
for analyzing the convergence behavior of an algorithm that
at iteration $k$, when the iterate $x_k$ is available,
uses the gradient  $\nabla F_k(x_k)$ to construct $x_{k+1}$,
as opposed to determining $x^*_k$ for each function $F_k$.
We illustrate this on a simple gradient-based method, given by
\begin{equation}\label{eq:gradmet}
x_{k+1} = x_k-s_k \nabla F_k(x_k)\qquad\hbox{for $k\ge 1$},
\end{equation}
where $x_1$ is an initial point and $s_k>0$ is a stepsize.
The idea behind the analysis of the method is resting on a relation of the form
$\|x_{k+1} - x^*\|\le q_k \|x_k - x^*\|+ r_k$
for some $q_k$ and $r_k$, and explore the conditions on $q_k$ and
$r_k$ that ensure the convergence
of $\|x_k-x^*\|$ to 0, as $k\to\infty$.

In what follows, we make use of the following result which can be found in~\cite{polyakbook},
Lemma 3 in Chapter 2.
\begin{lemma}\label{lema-polyak}
Let $\{u_k\}$ be a nonnegative scalar sequence such that
\[u_{k+1}\le q_k u_k+r_k\qquad\hbox{ for all $k$,}\]
where the scalars $q_k$ and $r_k$
satisfy the following conditions
\[\hbox{$q_k\in[0,1)$ and $r_k\ge 0$ for all $k$,}
\qquad \sum_{k=1}^\infty (1-q_k)=\infty,\qquad
\lim_{k\to\infty}\frac{r_k}{1-q_k}=0.\]
Then $u_k\to 0$ as $k\to\infty$.
\end{lemma}

With Lemma~\ref{lema-sols} and Lemma~\ref{lema-polyak} in place,
we next establish a set of conditions on $\{\g_k\}$ and $\{\d_k\}$
that ensure convergence of the iterates produced by the method~\eqref{eq:gradmet}.

\begin{proposition}\label{prop:gradconv}
Let $f$ be strongly convex with a constant $\mu>0$ and have Lipschitz continuous gradients
with a constant $L_f$.
Let the sequences $\{\g_k\}$ and $\{\d_k\}$ satisfy
\[\hbox{$\g_k>0$,\ $\d_k>0$, \ $\g_{k+1}\d_{k+1}\le \g_k\d_k$ \ \  for all $k$},
\qquad \lim_{k\to\infty}\g_k=+\infty.\]
Consider the method~\eqref{eq:gradmet} with the stepsize $s_k>0$ satisfying the conditions
\[\sum_{k=1}^\infty s_k=\infty,\qquad\sum_{k=1}^\infty s_k^2<\infty.\]
Moreover, assume that
\[\lim_{k\to\infty}s_k\g_k^2=0,\qquad
\lim_{k\to\infty}\frac{\g_k\d_k}{s_k^2}=0.\]
Then, the iterates $\{x_k\}$ the method~\eqref{eq:gradmet} converge to the solution $x^*$
of the original problem.
\end{proposition}
\begin{proof}
For any $k\ge0$, for the iterates of the method we have
\[\|x_{k+1}- x_k^*\|^2 =\|x_k - x_k^*\|^2 -2s_k\la \nabla F_k(x_k),x_k-x_k^*\ra +s_k^2\|\nabla F_k(x_k)\|^2.\]
By the strong convexity of $F_k$ and the fact $\nabla F_k(x_k^*)=0$, it follows that
\begin{equation}
\label{eq-main1}
\|x_{k+1}- x_k^*\|^2 \le (1-2s_k\mu)\|x_k - x_k^*\|^2 +s_k^2\|\nabla F_k(x_k)\|^2.
\end{equation}
For $\|\nabla F_k(x_k)\|^2$
we write
\[\|\nabla F_k(x_k)\|^2
\le 2\|\nabla f(x_k)\|^2 +2\left\|\frac{\g_k}{m} \sum_{i=1}^{m}
 \nabla h_k\left(x_k^*; a_i, b_i\right) \right\|^2
 \le 2\|\nabla f(x_k)\|^2 +2\g_k^2,
\]
where the last inequality is obtained by using the convexity of the squared-norm function
and the fact that $\|\nabla h_k\left(x; a_i, b_i\right) \|\le 1$ for any $x$ (see~\eqref{eq:pderiv}).
We further estimate $\|\nabla f(x_k)\|^2$ as follows:
\[\|\nabla f(x_k)\|^2\le 2\|\nabla f(x_k)-\nabla f(x^*)\|^2 +2\|\nabla f(x^*)\|^2
\le 2L_f^2\|x_k-x^*\|^2 +2\|\nabla f(x^*)\|^2,\]
where in the last inequality we use the Lipschitz gradient property of $f$.
Thus,
\[\|\nabla F_k(x_k)\|^2\le 4L_f^2\|x_k-x^*\|^2 +4\|\nabla f(x^*)\|^2 +2\g_k^2.\]
Further, we have
\[\|x_k-x^*\|^2\le 2\|x_k-x_k^*\|^2 +2\|x_k^*-x^*\|^2,\] so that
\[\|\nabla F_k(x_k)\|^2\le 8L_f^2\|x_k-x_k^*\|^2+ 8L_f^2\|x_k^*-x^*\|^2 +4\|\nabla f(x^*)\|^2 +2\g_k^2.\]
For large enough $k$, by Lemma~\ref{lema-sols} we have
\begin{equation}\label{eq-estimk}
\|x_k^*-x^*\|^2\le \frac{\g_k\d_k}{4\mu\a_{\min}}.\end{equation}
By combining the preceding two relations with relation~\eqref{eq-main1}
we obtain for all large enough $k$,
\begin{align*}
\|x_{k+1}- x_k^*\|^2
\le &(1-2s_k\mu +8L_f^2 s_k^2)\|x_k - x_k^*\|^2 \cr
&+s_k^2\left(\frac{2L_f^2   \g_k\d_k}{\mu\a_{\min}}+4\|\nabla f(x^*)\|^2 +2\g_k^2\right).
\end{align*}
We next consider $\|x_{k+1}- x_{k+1}^*\|^2$ for which we write
\[\|x_{k+1}- x_{k+1}^*\|^2 \le (1+s_k\mu)\|x_{k+1}- x_k^*\|^2 +(1+s_k^{-1}\mu^{-1})\|x_k^* - x_{k+1}^*\|^2.\]
Combining the preceding two relations, we obtain
\begin{align}\label{eq-main2}
\|x_{k+1}- x_k^*\|^2
\le &(1+s_k\mu) (1-2s_k\mu +8L_f^2 s_k^2)\|x_k - x_k^*\|^2 \cr
&+(1+s_k\mu)s_k^2\left(\frac{2L_f^2   \g_k\d_k}{\mu\a_{\min}}+4\|\nabla f(x^*)\|^2 +2\g_k^2\right)\cr
&+ (1+s_k^{-1}\mu^{-1})\|x_k^* - x_{k+1}^*\|^2.
\end{align}

Next we write
\[\|x_k^* - x_{k+1}^*\|^2\le 2\|x_k^* -x^*\|^2 +2\|x^*- x_{k+1}^*\|^2
\le 2\frac{\g_k\d_k}{4\mu\a_{\min}} + 2\frac{\g_{k+1}\d_{k+1}}{4\mu\a_{\min}}
\le \frac{\g_k\d_k}{\mu\a_{\min}},\]
where we use estimate~\eqref{eq-estimk}(cf.\ Lemma~\ref{lema-sols}) and the assumption that
$\g_{k+1}\d_{k+1}\le \g_k\d_k.$
Using the preceding estimate in~\eqref{eq-main2}, we obtain for large enough $k$,
\begin{align}\label{eq-main3}
\|x_{k+1}- x_{k+1}^*\|^2
\le &(1+s_k\mu) (1-2s_k\mu +8L_f^2 s_k^2)\|x_k - x_k^*\|^2 \cr
&+(1+s_k\mu)s_k^2\left(\frac{2L_f^2   \g_k\d_k}{\mu\a_{\min}}+4\|\nabla f(x^*)\|^2 +2\g_k^2\right)\cr
&+ (1+s_k^{-1}\mu^{-1})\frac{\g_k\d_k}{\mu\a_{\min}}.
\end{align}
The rest of the proof is verifyng that Lemma~\ref{lema-polyak} can be applied to the
preceding inequality with the following identification,
\[u_k=\|x_k - x_k^*\|^2,\qquad
q_k=(1+s_k\mu) (1-2s_k\mu +8L_f^2 s_k^2),\]
\[r_k=(1+s_k\mu)s_k^2\left(\frac{2L_f^2   \g_k\d_k}{\mu\a_{\min}}+4\|\nabla f(x^*)\|^2 +2\g_k^2\right)
+ (1+s_k^{-1}\mu^{-1})\frac{\g_k\d_k}{\mu\a_{\min}}.\]

Consider the coefficient $q_k$, for which we have
\begin{align*}
q_k
&=1-2s_k\mu +8L_f ^2s_k^2
+s_k\mu -2s_k^2\mu^2+ 8L_f^2 \mu s_k^3\cr
&= 1- s_k\mu +(8L_f^2 -2\mu^2)s_k^2
+ 8L_f^2 \mu s_k^3\cr
&\ge 1- s_k\mu,\end{align*}
where in the last inequality we use $L_f\ge \mu$.
For large enough $k$, since $s_k\to0$, the coefficient $q_k$ is positive.
Also, since $s_k\to0$, we have for large enough $k$
\[q_k= 1- s_k\mu +(8L_f^2 -2\mu^2)s_k^2
+ 8L_f^2 \mu s_k^2\le 1-s_k\mu +16L_f^2 \mu s_k^2.\]
Hence, $s_k\ge 1-q_k\ge s_k\mu-16 L_f^2\mu s_k^2,$
implying that
\[\sum_{k=1}^\infty 1-q_k=+\infty,\]
since $\sum_{k=1}^\infty s_k=+\infty$ and $\sum_{k=1}^\infty s^2_k<+\infty$.
Moreover,
\[\frac{1}{1-q_k}\approx \Theta\left(\frac{1}{s_k}\right).\]
For the coefficient $r_k$, since $\g_k\d_k$ is nonincreasing, $s_k\to 0$, and $\g_k\to\infty$
we have for large enough $k$,
\[r_k\approx O(s_k^2\g_k^2) +O\left(\frac{\g_k\d_k}{s_k}\right).\]
From the preceding two relations it follows that
\[\frac{r_k}{1-q_k}\approx O(s_k\g_k^2) +O\left(\frac{\g_k\d_k}{s_k^2}\right).\]
In view of the conditions $s_k\g_k^2\to0$ and $\frac{\g_k\d_k}{s_k^2}\to0$,
it follows that
\[\lim_{k\to\infty}\frac{r_k}{1-q_k}=0.\]
Thus, all the conditions of Lemma~\ref{lema-polyak} are satisfied and we conclude that
\[\lim_{k\to\infty}\|x_k -x_k^*\|= 0.\]
Finally, since $\frac{\g_k\d_k}{s_k^2}\to0$ and $s_k\to0$, it follows that $\g_k\d_k\to0$,
which in view of Lemma~\ref{lema-sols}, implies that
\[\lim_{k\to\infty}\|x_k^* - x^*\|= 0.\]
The preceding two relations yield $\lim_{k\to\infty}\|x_k - x^*\|= 0$.
\end{proof}

We next discuss a possible choice of the penalty parameters and the stepsize
that satisfy the conditions of Proposition~\ref{prop:gradconv}.
Let stepsize be of the form
\[s_k=\frac{1}{k^{\frac{1}{2}+\e} }\qquad\hbox{with $0<\e \le \frac{1}{2}$}.\]
This stepsize satisfies the conditions $\sum_{k=1}^\infty s_k=\infty$ and
$\sum_{k=1}^\infty s_k^2<\infty$
of Proposition~\ref{prop:gradconv}.
Next, let
\[\g_k=\ln (k+1)\qquad\hbox{for all }k\ge 1.\]
Thus, we have $\g_k\to\infty$ as required in Proposition~\ref{prop:gradconv}.
Moreover, we have
\[\lim_{k\to\infty}s_k\g_k^2 =0.\]
It remains to choose $\d_k$ so that $\g_k\d_k$ is nonincreasing and
$\lim_{k\to\infty}\frac{\g_k\d_k}{s_k^2}=0.$
Let
\[\d_k=\frac{1}{k^b}\qquad\hbox{for all $k$ and some for $b>0$},\]
which ensures that $\g_k\d_k$ is nonincreasing.
We then have
\[\frac{\g_k\d_k}{s_k^2}=\ln(k+1) k^{1+2\e- b},\]
which would tend to 0, as long as $1+2\e- b<0$.

While the conditions on the penalty parameters and the stepsize choice in
Proposition~\ref{prop:gradconv} can be satisfied, it would be of interest to understand what
is the best convergence rate that can be achieved. If the method~\eqref{eq:gradmet} is implemented
in an incremental
manner, there will be additional error in the order of the squared stepsize value.
Thus, it is expected that the incremental implementation
would have the same rate as that of the method~\eqref{eq:gradmet}.
}

\section{Conclusion}\label{sec:concl}
In this paper, we provided a novel penalty re-formulation
for a convex minimization problem with linear constraints.
The structure of the penalty functions that we used to penalize the linear constraints,
and the suitable choices of the penalty parameters
render the penalized unconstrained
problem with solutions that are \emph{feasible} for the original constrained problem.
In addition, with an additional constraint on the penalty parameters imposed by a desired accuracy level,
the solutions of the penalized unconstrained
problem are guaranteed to be arbitrarily close to the solution set of the original problem.
An advantage of the proposed penalty reformulation is in the ability to
employ fast incremental gradient methods, such as SAGA.
\tat{However, this reformulation requires some apriori knowledge about some properties of the problem under consideration.
To rectify this issue, we proposed an alternative approach to set up a time-dependent penalized problem. We demonstrated that under an appropriate choice of the time-varying penalty parameters one can apply the single-loop gradient based procedure and achieve its convergence to the optimum given a strongly convex objective function.}

\section*{Acknowledgement}
The authors are thankful to the reviewers for providing valuable comments that have
helped us improve the paper substantially.
A.~Nedi\'c gratefully acknowledges a support of this work by the
Office of Naval Research grant no.\ N00014-12-1-0998.

 \bibliographystyle{abbrv}
\bibliography{Literature}

\end{document}

%% file: ex_shared.tex

\ifpdf
  \DeclareGraphicsExtensions{.eps,.pdf,.png,.jpg}
\else
  \DeclareGraphicsExtensions{.eps}
\fi


\newsiamremark{remark}{Remark}
\newsiamremark{hypothesis}{Hypothesis}
\crefname{hypothesis}{Hypothesis}{Hypotheses}
\newsiamthm{claim}{Claim}

\headers{A Smooth Inexact Penalty Reformulation of Convex Problems}{T. Tatarenko and A. Nedi\'c}

\title{A Smooth Inexact Penalty Reformulation \\ of Convex Problems 
with Linear Constraints\thanks{Submitted to the editors on 22.08.2018}}

\author{Tatiana Tatarenko\thanks{Department of Control Theory and Robotics, TU Darmstadt, Germany
  (\email{tatarenk@rmr.tu-darmstadt.de}).}
\and Angelia Nedi\'c\thanks{The School of Electrical, Computer and Energy Engineering at Arizona State University, USA
  (\email{Angelia.Nedich@asu.edu}).}
}

